\documentclass[11pt,oneside,a4paper]{article}
\usepackage{hyperref}
\usepackage{lipsum}
\usepackage{amsfonts}
\usepackage{graphicx}
\usepackage{epstopdf}
\usepackage{amssymb}  
\usepackage{amsthm}
\usepackage{amsmath}

\usepackage{geometry}
\usepackage[french]{layout}
\usepackage{fancyhdr} 
\title{Approximate and exact controllability of the continuity equation with a localized vector field\thanks{This work has been carried out in the framework of Archim\`ede Labex (ANR-11-LABX-0033) and of the A*MIDEX project (ANR-11-IDEX-0001-02), funded by the ``Investissements d'Avenir" French Government programme managed by the French National Research Agency (ANR). The authors acknowledge the support of the ANR project CroCo ANR-16-CE33-0008.}}

\author{
  Michel Duprez\thanks{Aix Marseille Universit\'e, CNRS, Centrale Marseille, I2M,  Marseille, France.   (\href{mailto:mduprez@math.cnrs.fr}{mduprez@math.cnrs.fr}), Corresponding author.}
  \and
  Morgan Morancey\thanks{Aix Marseille Universit\'e, CNRS, Centrale Marseille, I2M, Marseille, France. (\href{mailto:morgan.morancey@univ-amu.fr}{morgan.morancey@univ-amu.fr}).}
  \and
  Francesco Rossi\thanks{Dipartimento di Matematica ``Tullio Levi-Civita", Universit\`{a} degli Studi di Padova, Via Trieste 63, 35121 Padova, Italy. (\href{mailto:{francesco.rossi@math.unipd.it}}{francesco.rossi@math.unipd.it}). }
}

\usepackage{amsopn}

\newcommand{\Supp}{\operatorname{supp}} 
\newcommand{\supp}{\operatorname{supp}}

\usepackage{tikz}
\usetikzlibrary{patterns}
\usetikzlibrary{arrows}
\tikzset{cross/.style={path picture={
  \draw[black]
    (path picture bounding box.south east)--(path picture bounding box.north west)
    (path picture bounding box.south west)--(path picture bounding box.north east);
}}}

\newcommand{\storto}[5]{
\pgfmoveto{\pgfxy(#1)}
\pgfcurveto{\pgfxy(#2)}{\pgfxy(#2)}{\pgfxy(#3)}
\pgfmoveto{\pgfxy(#3)}
\pgfcurveto{\pgfxy(#4)}{\pgfxy(#4)}{\pgfxy(#5)}
\pgfstroke}

\newcommand{\Div}{\operatorname{div}} 

\newcommand{\mb}[1]{\mathbb{#1}}
\newcommand{\mc}[1]{\mathcal{#1}}

\newcommand{\mr}[1]{\mathrm{#1}}

\usepackage{dsfont}

\theoremstyle{definition} 
\newtheorem{definition}{\textsc{Definition}}
\newtheorem{cond}{\textsc{Condition}}[section]

\theoremstyle{remark}
\newtheorem{rmq}{\textsf{Remark}}

\theoremstyle{plain}
\newtheorem{theorem}{\textsc{Theorem}}[section]
\newtheorem{lemma}{\textsc{Lemma}}[section]
\newtheorem{prop}{\textsc{Proposition}}[section]
\newtheorem{propi}{\textsf{Property}}[section]

\usepackage{amssymb}

\begin{document}

\maketitle

\begin{abstract}
We study controllability of a Partial Differential Equation of transport type, that arises in crowd models. We are interested in controlling it with a control being a vector field, representing a perturbation of the velocity,  localized on a fixed control set.

We prove that, for each initial and final configuration, one can steer approximately one to another with Lipschitz controls when the uncontrolled dynamics allows to cross the control set.
We also show that the exact controllability only holds for controls with less regularity,
for which one may lose uniqueness of the associated solution.

\end{abstract}

%

\section{Introduction}

In recent years, the study of systems describing a crowd of interacting autonomous agents 
has drawn a great interest from the control community (see \textit{e.g.} the Cucker-Smale model \cite{CS07}).
A better understanding of such interaction phenomena can have a strong impact in several key applications, such as road traffic and egress problems for pedestrians. For a few reviews about this topic, see \textit{e.g.} \cite{axelrod,active1,camazine,CPTbook,helbing,jackson,MT14,SepulchreReview}.

Beside the description of interactions, it is now relevant to study problems of {\bf control of crowds}, \textit{i.e.} of controlling such systems by acting on few agents, or on the crowd localized in a small subset of the configuration space. The nature of the control problem relies on the model used to describe the crowd. Two main classes are widely used. 

In {\bf microscopic models}, the position of each agent is clearly identified; the crowd dynamics is described by a large dimensional ordinary differential equation, in which couplings of terms represent interactions. For control of such models, a large literature is available from the control community, under the generic name of networked control (see \textit{e.g.} \cite{bullo,kumar,lin}). There are several control applications to pedestrian crowds \cite{ferscha,luh} and  road traffic \cite{canudas,hegyi}.

In {\bf macroscopic models}, instead, the idea is to represent the crowd by the spatial density of agents; in this setting, the evolution of the density solves a partial differential equation of transport type. 
Nonlocal terms (such as convolution) model the interactions between the agents. 
In this article, we focus on this second approach,
\textit{i.e.} macroscopic models. 
%
To our knowledge, there exist few studies of control of this family of equations.
In \cite{PRT15}, the authors provide approximate alignment of a crowd described by the 
macroscopic Cucker-Smale model \cite{CS07}. The control is the acceleration, and it is localized in a control region $\omega$ which moves in time. In a similar situation, a stabilization strategy has been established in  \cite{CPRT17,CPRT17b}, by generalizing the Jurdjevic-Quinn method to partial differential equations.

A different approach is given by mean-field type control, \textit{i.e.} control of mean-field equations and of mean-field games modeling crowds. See \textit{e.g.} \cite{achdou2,achdou1,carmona,FS}. In this case, problems are often of optimization nature, \textit{i.e.} the goal is to find a control minimizing a given cost. In this article, we are mainly interested in controllability 
 problems, for which mean-field type control approaches seem not adapted.


In this article, we study a macroscopic model, thus the crowd is represented by its density, that is a time-evolving measure $\mu(t)$ defined for  positive  times $t$ on the space $\mb{R}^d$ ($d\geq 1$).
The natural (uncontrolled) velocity field for the measure is denoted by  $v: \mb{R}^d\rightarrow\mb{R}^d$, being a vector field assumed Lipschitz and uniformly bounded.
We act on the velocity field in a fixed portion  $\omega$ of the space, 
which will be a  {\bf nonempty open connected subset} of $\mb{R}^d$. 
The admissible controls are thus functions of the form $\mathds{1}_{\omega}u:\mb{R}^d\times\mb{R}^+\rightarrow\mb{R}^d$.
We then consider the following  linear transport equation 
\begin{equation}\label{eq:transport}
	\left\{
	\begin{array}{ll}
\partial_t\mu +\nabla\cdot((v+\mathds{1}_{\omega}u)\mu)=0&\mbox{ in }\mb{R}^d\times\mb{R}^+,\\\noalign{\smallskip}
\mu(0)=\mu^0&\mbox{ in }\mb{R}^d,\\
	\end{array}
	\right.
\end{equation}
where 
$\mu^0$ is the initial data (initial configuration of the crowd) and the  function $u$ is an admissible control. 
The function $v+\mathds{1}_{\omega}u$ represents the velocity field acting on $\mu$.
System \eqref{eq:transport} is a first simple approximation for crowd modelling, since the uncontrolled vector field $v$ is given, and it does not describe interactions between agents. Nevertheless, it is necessary to understand controllability properties for such simple equation
as a first step, before dealing with velocity fields depending on the crowd itself. 
Thus, in a future work, we will study controllability of crowd models with a nonlocal term $v[\mu]$, based on the linear results presented here. 

Even though System \eqref{eq:transport} is linear, the control acts on the velocity, 
thus the control problem is nonlinear, which is one of the main difficulties in this study.

\smallskip


The goal of this work is to study the control properties of System \eqref{eq:transport}.
We now recall the notion of approximate controllability and exact controllability for System \eqref{eq:transport}. 
We say that System \eqref{eq:transport} is {\it approximately controllable}
from $\mu^0$ to $\mu^1$ on the time interval $[0,T]$ if 
we can steer the solution to  System \eqref{eq:transport} at time $T$ as close to $\mu^1$ as we want with an appropriate control $\mathds{1}_{\omega}u$.
Similarly, we say that 
System \eqref{eq:transport} is  {\it exactly controllable}
from $\mu^0$ to $\mu^1$ 
on the time interval $[0,T]$  if we can steer the solution to  System \eqref{eq:transport} at time $T$ exactly to $\mu^1$ with an appropriate control $\mathds{1}_{\omega}u$.
In Definition \ref{def:approx} below, we give a formal definition of the notion of
 approximate controllability in terms of Wasserstein distance. 

\smallskip

The main results of this article show that approximate and exact controllability depend on two main aspects:
first, from a geometrical point of view, the uncontrolled vector field $v$ needs to send the support of $\mu^0$ to $\omega$ forward in time and the support of $\mu^1$ to $\omega$ backward in time. 
This idea is formulated in the following Condition:

\begin{cond}[Geometrical condition]\label{cond1}
Let $\mu^0,\mu^1$ be two probability measures on $\mb{R}^d$ satisfying:
\begin{enumerate}
\item[(i)] For each $x^0\in \supp(\mu^0)$, 
there exists $t^0>0$ such that $\Phi_{t^0}^v(x^0)\in \omega,$
where $\Phi_{t}^v$ is the \textit{flow} associated to $v$, \textit{i.e.} the solution to the Cauchy problem 
\begin{equation*}
\left\{\begin{array}{l}
\dot x(t) =v(x(t))\mbox{ for a.e. }t>0,\\\noalign{\smallskip}
x(0)=x^0.
\end{array}\right.
\end{equation*}

\item[(ii)] For each $x^1\in \supp(\mu^1)$, 
there exists $t^1>0$ such that $\Phi_{-t^1}^{v}(x^1)\in \omega$.
\end{enumerate}
\end{cond}

This geometrical aspect is illustrated in Figure \ref{fig:cond geo}.

\begin{figure}[htb]
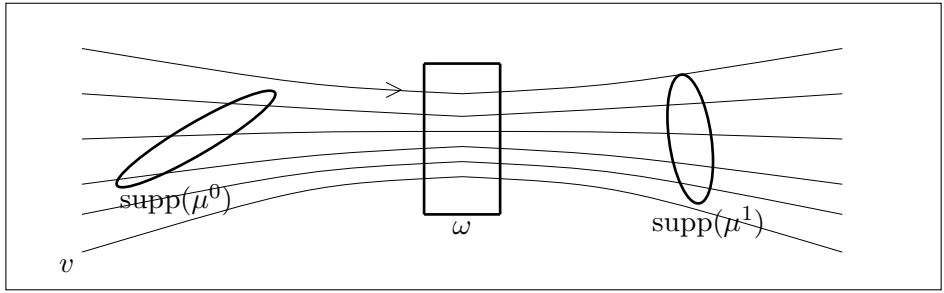


\begin{center}
\begin{pgfpictureboxed}{0cm}{0cm}{12.3 cm}{3.8cm}

\begin{pgfscope}
\pgfsetstartarrow{\pgfarrowswap{\pgfarrowto}}
\pgfsetendarrow{\pgfarrowto}
\pgfsetlinewidth{.2pt}
\storto{1,0.5}{4,1.4}{6,2-0.5}{8,1.9-0.5}{11,1-0.5}
\storto{1,1.5-0.5}{4,2.1-0.5}{6,2.2-0.5}{8,2.1-0.5}{11,1.5-0.5}
\storto{1,1.9-0.5}{4,2.3-0.5}{6,2.4-0.5}{8,2.3-0.5}{11,1.9-0.5}
\storto{1,2.5-0.5}{4,2.6-0.5}{6,2.6-0.5}{8,2.6-0.5}{11,2.5-0.5}
\storto{1,3.1-0.5}{4,2.9-0.5}{6,2.8-0.5}{8,2.9-0.5}{11,3.1-0.5}
\storto{1,3.7-0.5}{4,3.2-0.5}{6,3.1-0.5}{8,3.2-0.5}{11,3.7-0.5}
\pgfline{\pgfxy(5,2.75)}{\pgfxy(5.2,2.65)}
\pgfline{\pgfxy(5,2.55)}{\pgfxy(5.2,2.65)}
\end{pgfscope}

\begin{pgfscope}
\pgfsetlinewidth{1pt}
\pgfline{\pgfxy(5.5,1)}{\pgfxy(5.5,3)}
\pgfline{\pgfxy(6.5,1)}{\pgfxy(6.5,3)}
\pgfline{\pgfxy(5.5,1)}{\pgfxy(6.5,1)}
\pgfline{\pgfxy(5.5,3)}{\pgfxy(6.5,3)}
\pgfellipse[stroke]{\pgfxy(2.5,2)}{\pgfxy(0.3,0.4)}{\pgfxy(1,0.5)}
\pgfellipse[stroke]{\pgfxy(9,2)}{\pgfxy(0,.8)}{\pgfxy(-.3,0.3)}
\end{pgfscope}

\pgfputat{\pgfxy(1.5,1.4)}{\pgfbox[left,top]{$\supp(\mu^0)$}}
\pgfputat{\pgfxy(6,.9)}{\pgfbox[center,top]{$\omega$}}
\pgfputat{\pgfxy(8.5,1.1)}{\pgfbox[left,top]{$\supp(\mu^1)$}}
\pgfputat{\pgfxy(0.7,0.4)}{\pgfbox[left,top]{$v$}}
\end{pgfpictureboxed}

\caption{Geometrical condition.}
\label{fig:cond geo}
\end{center}
\end{figure}

\begin{rmq}
Condition \ref{cond1} is the minimal one that we can expect to steer any initial condition to 
any target.
Indeed, if there exists a point $x^0$ of the interior of $\supp(\mu^0)$
for which the first item of Condition \ref{cond1} is not satisfied, 
then there exists a whole subpopulation of the measure $\mu^0$ that never intersects the control region,
thus we cannot act on it. 
\end{rmq}

The second aspect that we want to highlight is the following: 
The measures $\mu^0$ and $\mu^1$ need to be sufficiently regular  with
respect to the flow generated by $v+\mathds{1}_{\omega}u$. Three cases are particularly relevant: 

\bigskip

\textbf{a) Controllability with Lipschitz controls}\\
 If we impose the classical Carath\'eodory condition of $\mathds{1}_{\omega}u$ being Lipschitz in space, measurable in time and uniformly bounded, 
 then the flow $\Phi^{v+\mathds{1}_{\omega}u}_t$ is an homeomorphism (see \cite[Th. 2.1.1]{BP07}).
As a result, one can expect  approximate 
controllability only, since for general measures there exists no homeomorphism sending one to another.
For more details, see Section \ref{sec:non exact cont}.
We then have the following result:

\begin{theorem}[Main result - Controllability with Lipschitz control]\label{th cont approx}
Let  $\mu^0$, $\mu^1$ be two probability measures on $\mb{R}^d$ 
compactly supported, absolutely continuous with respect to the Lebesgue measure and satisfying Condition \ref{cond1}. 
Then there exists $T$ such that System \eqref{eq:transport} is 
\textbf{approximately controllable} on the time interval $[0,T]$
from $\mu^0$ to $\mu^1$
with a control $\mathds{1}_{\omega}u:\mb{R}^d\times\mb{R}^+\rightarrow\mb{R}^d$ uniformly bounded, Lipschitz in space and  measurable in time.
\end{theorem}

We give a proof of Theorem \ref{th cont approx} in Section \ref{sec:cont approx}.

\bigskip
\textbf{b) Controllability with BV controls}\\
To hope to obtain  exact controllability of System \eqref{eq:transport}, 
it is then necessary to search among controls $\mathds{1}_{\omega}u$ with less regularity. 
A weaker condition on the regularity of the velocity field 
for the well-posedness of System \eqref{eq:transport} has been given in \cite{A04} for  vector fields $v$ with bounded variations (we will use the abbreviation BV vector fields) satisfying in particular
\vspace*{-2mm}\begin{equation}\label{cond ambr}
\displaystyle\int_0^T\|[\Div v]^-\|_{L^{\infty}(\mb{R}^d)}<+\infty.
\end{equation}
As it will be explained in Section \ref{sec:non exact cont},
if we choose the admissible controls satisfying the  setting of \cite{A04},
then it is not necessary that the support of $\mu^0$ and $\mu^1$ are homeomorph
for exact controllability, but it is not sufficient either (see Section \ref{sec:non exact cont}).
For example,
Condition \eqref{cond ambr} allows to steer a measure which support is connected 
to a measure which support is composed of two connected components, 
but the inverse is forbidden.
Thus, even this setting does not allow to yield exact controllability.

\bigskip

\textbf{c) Controllability with Borel controls}\\
We then consider an even larger class of controls, that are Borel vector fields. In this setting, we have exact controllability under the geometrical Condition \ref{cond1}. The main drawback is that, in this less regular setting, System \eqref{eq:transport} is necessarily not well-posed. In particular, one has not necessarily uniqueness of the solution. 
For this reason, one needs to describe solutions to System \eqref{eq:transport} as pairs $(\mathds{1}_{\omega} u, \mu)$, where $\mu$ is one among the admissible solutions with control $\mathds{1}_{\omega}u$. 



\begin{theorem}[Main result - Controllability with Borel control]\label{th cont exact}
Let  $\mu^0,\mu^1$ be two probability measures on $\mb{R}^d$ 
compactly supported and satisfying Condition \ref{cond1}. 
Then, there exists $T>0$ such that System \eqref{eq:transport} is \textbf{exactly controllable} on the time interval $[0,T]$
from $\mu^0$ to $\mu^1$
in the following sense: there exists a couple $(\mathds{1}_{\omega}u,\mu)$ composed of  
a Borel vector field $\mathds{1}_{\omega}u:\mb{R}^d\times\mb{R}^+\rightarrow\mb{R}^d$ 
and  a time-evolving measure $\mu$
being  weak  solution to  System \eqref{eq:transport} (see Definition \ref{def:weak})
 and satisfying  
 \vspace*{-2mm}\begin{equation*}
 \mu(T)=\mu_1.
 \end{equation*}
\end{theorem}

A proof of Theorem \ref{th cont exact} is given in Section \ref{sec:exact cont}.


\smallskip

This paper is organised as follows. 
In Section \ref{section 2}, we recall basic properties of the Wasserstein distance and the continuity equation.
Section \ref{sec:cont approx} is devoted to the proof of  Theorem  \ref{th cont approx}, \textit{i.e.} the approximate controllability of System \eqref{eq:transport} with a Lipschitz localized vector field.
Finally, in Section \ref{sec:exact cont}, 
we first show that exact controllability does not hold for Lipschitz or BV controls; 
we also  prove Theorem \ref{th cont exact}, \textit{i.e.} exact controllability of System \eqref{eq:transport} with a Borel localized vector field.

\section{The Wasserstein distance and the  continuity equation}\label{section 2}

In this section, we recall the definition and some properties of the  Wasserstein distance 
and the continuity equation,
which will be used all along this paper. 
We denote by $\mc{P}_c(\mb{R}^d)$ the space of probability measures in $\mb{R}^d$ with compact support
and for $\mu,~\nu\in\mc{P}_c(\mb{R}^d)$, we denote by $\Pi(\mu,\nu)$ the set of \textit{transference plans}
 from $\mu$ to $\nu$, \textit{i.e.} the probability measures on $\mb{R}^d\times\mb{R}^d$ 
 satisfying
\begin{equation*}
 \int_{\mb{R}^d}d\pi(x,\cdot)=d\mu(x)
 \mbox{ and }\int_{\mb{R}^d}d\pi(\cdot,y)=d\nu(y).
\end{equation*}

\begin{definition}
Let $p\in[1,\infty)$ and $\mu,\nu\in \mc{P}_c(\mb{R}^d)$. Define 
\begin{equation}\label{def:wass plan}
W_p(\mu,\nu)=\inf\limits_{\pi\in\Pi(\mu,\nu)}\left\{\left(\displaystyle\iint_{\mb{R}^d\times\mb{R}^d}
|x-y|^pd\pi\right)^{1/p}\right\}.
\end{equation}
The quantity is called the \textbf{Wasserstein distance}.
\end{definition}

This is the idea of \textit{optimal transportation},
consisting  in finding the optimal way to transport
mass from a given measure to another.
For a thorough introduction, see \textit{e.g.} \cite{V03}.

We denote by $\Gamma$  the set of Borel maps $\gamma:\mb{R}^d\rightarrow\mb{R}^d$.
We now recall the definition of the \textit{push-forward} of a measure:
\begin{definition}
For a $\gamma\in\Gamma$, 
we define the {\it push-forward} $\gamma\#\mu$ of a measure $\mu$ of $\mb{R}^d$ as follows:
\vspace*{-2mm}\begin{equation*}
(\gamma\#\mu)(E):=\mu(\gamma^{-1}(E)),
\end{equation*}
for every subset $E$ such that $\gamma^{-1}(E)$ is $\mu$-measurable.
\end{definition}

We denote by ``AC measures''
the measures which are absolutely continuous with respect to the Lebesgue measure
and by $\mc{P}_c^{ac}(\mb{R}^d)$  the subset  of $\mc{P}_c(\mb{R}^d)$ of AC measures. 
On $\mc{P}_c^{ac}(\mb{R}^d)$, the Wasserstein distance can be reformulated as follows:
\begin{propi}[{see \cite[Chap. 7]{V03}}]
Let $p\in[1,\infty)$ and $\mu,\nu\in \mc{P}^{ac}_c(\mb{R}^d)$. It holds
\begin{equation}\label{def:Wp}
W_p(\mu,\nu)=\inf\limits_{\gamma\in\Gamma}\left\{\left(\displaystyle\int_{\mb{R}^d}
|\gamma(x)-x|^pd\mu\right)^{1/p}:\gamma\#\mu=\nu\right\}.
\end{equation}
\end{propi}
The Wasserstein distance satisfies some useful properties:
\begin{propi}[{see \cite[Chap. 7]{V03}}]\label{prop Wp}
Let $p\in[1,\infty)$.
\begin{enumerate}
\item[(i)]
The  Wasserstein distance $W_p$ is a  distance on $\mc{P}_c(\mb{R}^d)$.
\item [(ii)] The topology induced by the Wasserstein distance $W_p$ on $\mc{P}_c(\mb{R}^d)$ coincides with the weak topology.
\item[(iii)]For all $\mu,\nu\in \mc{P}_c^{ac}(\mb{R}^d)$, the infimum in \eqref{def:Wp}
 is achieved by at least one minimizer.
\end{enumerate}
\end{propi}

The Wasserstein distance can be extended to all pairs of measures $\mu,\nu$ 
compactly supported with the same total mass $\mu(\mb{R}^d)=\nu(\mb{R}^d)\neq0$, by the formula
$$W_p(\mu,\nu)=\mu(\mb{R}^d)^{1/p} W_p\left(\frac{\mu}{\mu(\mb{R}^d)},\frac{\nu}{\nu(\mb{R}^d)}\right).$$


In the rest of the paper, the following properties of the Wasserstein distance will be also helpful:
\begin{propi}[see \cite{PR13,V03}]
Let $\mu,~\rho,~\nu,~\eta$ be four positive measures compactly supported satisfying $\mu(\mb{R}^d)=\nu(\mb{R}^d)$ and $\rho(\mb{R}^d)=\eta(\mb{R}^d)$.
\begin{enumerate}
\item[(i)]
 For each  $p\in[1,\infty)$, it holds
\begin{equation}\label{ine wasser}
W^p_p(\mu+\rho,\nu+\eta)
\leqslant W^p_p(\mu,\nu)+W^p_p(\rho,\eta).
\end{equation}
\item[(ii)]
For each   $p_1,~p_2\in[1,\infty)$ with $p_1\leqslant p_2$, it holds
\begin{equation}\label{ine wasser3}
\left\{\begin{array}{l}
W_{p_1}(\mu,\nu)\leqslant W_{p_2}(\mu,\nu),\\\noalign{\smallskip}
W_{p_2}(\mu,\nu)
\leqslant \mr{diam}(X)^{1-p_1/p_2}W_{p_1}^{p_1/p_2}(\mu,\nu),
\end{array}\right.
\end{equation}
where $X$ contains the supports of $\mu$ and $\nu$.
\end{enumerate}
\end{propi}


We now recall the definition of  the continuity equation and the associated notion of weak solutions:
\begin{definition}\label{def:weak}
Let $T>0$ and $\mu^0$ be a measure in $\mb{R}^d$. 
We said that a pair  $(\mu,w)$ composed with a measure $\mu$ 
in $\mb{R}^d\times [0,T]$ and a vector field $w:\mb{R}^d\times\mb{R}^+\rightarrow \mb{R}^d$ satisfying 
\vspace*{-2mm}\begin{equation*}
\displaystyle\int_{0}^T\int_{\mb{R}^d}|w(t)|~d\mu(t)dt<\infty
\vspace*{-2mm}\end{equation*}
is a \textbf{weak solution} to the system, called the \textbf{continuity equation},
\begin{equation}\label{eq:transport sec 2}
	\left\{
	\begin{array}{ll}
\partial_t\mu +\nabla\cdot(w\mu)=0&\mbox{ in }\mb{R}^d\times[0,T],\\\noalign{\smallskip}
\mu(0)=\mu^0&\mbox{ in }\mb{R}^d,
	\end{array}
	\right.
\end{equation}
if for every continuous bounded function $\xi:\mb{R}^d\rightarrow\mb{R}$, the function 
$t\mapsto \int_{\mb{R}^d}\xi~d\mu(t)$ is absolutely continuous with respect to $t$
and for all $\psi\in\mc{C}^{\infty}_c(\mb{R}^d)$, it holds
\vspace*{-2mm}\begin{equation*}
\dfrac{d}{dt}\displaystyle\int_{\mb{R}^d}\psi ~d\mu(t)
=\displaystyle\int_{\mb{R}^d}\langle\nabla \psi,w(t)\rangle~ d\mu(t)
\vspace*{-2mm}\end{equation*}
 for \textit{a.e.} $t$ and $\mu(0)=\mu^0$.
\end{definition}
Note that $t\mapsto \mu(t)$ is continuous for the weak convergence, it then make sense to impose the 
initial condition $\mu(0)=\mu^0$ pointwisely in time. 
Before stating a result of existence and uniqueness of solutions for the continuity
equation,  we first recall the definition of the flow associated to a vector field.
\begin{definition}
Let $w:\mb{R}^d\times\mb{R}^+\rightarrow\mb{R}^d$ be a vector field being uniformly bounded, Lipschitz in space and  measurable in time.
We define the \textbf{flow} associated to the vector field $w$ 
as the application $(x^0,t)\mapsto\Phi_t^w(x^0)$ such that, for all $x^0\in\mb{R}^d$, 
$t\mapsto\Phi_t^w(x^0)$ is the solution to the Cauchy problem
\vspace*{-2mm}\begin{equation*}
\left\{\begin{array}{l}
\dot x(t) =w(x(t),t)\mbox{ for a.e. }t\geqslant 0,\\\noalign{\smallskip}
x(0)=x^0.
\end{array}\right.
\vspace*{-2mm}\end{equation*}
\end{definition}
The following property of the flow will be useful all along the present paper:
\begin{propi}[see \cite{PR13}]
Let $\mu,~\nu\in\mc{P}_c(\mb{R}^d)$
and $w:\mb{R}^d\times\mb{R}\rightarrow\mb{R}^d$ be a vector field uniformly bounded, 
Lipschitz in space and measurable in time 
with a Lipschitz constant equal to $L$.
 For each $t\in\mb{R}$
and  $p\in[1,\infty)$, it holds
\begin{equation}\label{ine wasser 2}
W_p(\Phi_t^w\#\mu,\Phi_t^w\#\nu)
\leqslant e^{\frac{(p+1)}{p}L|t|} W_p(\mu,\nu).
\end{equation}
\end{propi}

 We now recall a standard result for the continuity equation: 
 \begin{theorem}[see {\cite[Th. 5.34]{V03}}]
Let $T>0$,  $\mu^0\in \mc{P}_c(\mb{R}^d)$ and  $w$ a vector field uniformly bounded, Lipschitz in space and measurable in time.
Then, System \eqref{eq:transport sec 2}
admits a unique solution $\mu$ in $\mc{C}^0([0,T];\mc{P}_c(\mb{R}^d))$, 
where $\mc{P}_c(\mb{R}^d)$ is equipped with the weak topology. Moreover:
\begin{enumerate}
\item[(i)]If  $\mu^0\in \mc{P}_c^{ac}(\mb{R}^d)$, then the solution $\mu$ to  \eqref{eq:transport sec 2} 
belongs to $\mc{C}^0([0,T];\mc{P}_c^{ac}(\mb{R}^d))$.
\item[(ii)] We have $\mu(t)=\Phi_t^{w}\#\mu^0$ for all $t\in [0,T]$.
\end{enumerate}
 \end{theorem}

We now recall the precise notions of approximate controllability and exact controllability for System \eqref{eq:transport}:
\begin{definition}\label{def:approx}
We say that:
\begin{enumerate}
\item[$\bullet$] System \eqref{eq:transport} is 
{\bf approximately controllable}
from $\mu^0$ to $\mu^1$ on the time interval $[0,T]$ if 
for each $\varepsilon>0$ 
there exists a control $\mathds{1}_{\omega}u$  
such that the corresponding 
solutions $\mu$ to System \eqref{eq:transport} satisfies
\begin{equation}\label{estim Wp approx bis}
W_p(\mu^1,\mu(T))\leqslant \varepsilon.
\end{equation}
\item[$\bullet$] System \eqref{eq:transport} is  {\bf exactly controllable}
from $\mu^0$ to $\mu^1$ on the time interval $[0,T]$ if there exists a control $\mathds{1}_{\omega}u$ 
such that the corresponding solution  to System \eqref{eq:transport} is equal to $\mu^1$ at time $T$.
\end{enumerate}
\end{definition}

It is  interesting to remark that, by using  properties \eqref{ine wasser3} of the Wasserstein distance, estimate \eqref{estim Wp approx bis} 
can be replaced by:
\begin{equation*}
W_1(\mu^1,\mu(T))\leqslant \varepsilon.
\end{equation*}
Thus, in this work, we study  approximate controllability by considering the distance $W_1$ only.

\section{Approximate controllability with a localized Lipschitz control}\label{sec:cont approx}

In this section,  we study approximate controllability of System \eqref{eq:transport}
 with localized Lipschitz controls. 
More precisely, 
in  Sections 
 \ref{sec:approx dimn},
we consider the case where the open connected control subset
 $\omega$ contains the support of both $\mu^0$ and $\mu^1$.
We then prove Theorem \ref{th cont approx} in Section \ref{sec:approx loc}.

\subsection{Approximate controllability  with a Lipschitz control}\label{sec:approx dimn}

In this section, we prove  approximate controllability of System \eqref{eq:transport} with 
a Lipschitz control, when the open connected control subset $\omega$ contains the support of both $\mu^0$ and $\mu^1$.
 Without loss of generality, we can assume that the vector field $v$ is identically zero by replacing $u$ with $u-v$ in the control set $\omega$. 
We then study  approximate controllability of system 
\begin{equation}\label{eq:transport dim2}
	\left\{
	\begin{array}{ll}
\partial_t\mu +\Div(u\mu)=0&\mbox{ in }\mb{R}^d\times\mb{R}^+,\\\noalign{\smallskip}
\mu(0)=\mu^0&\mbox{ in }\mb{R}^d.
	\end{array}
	\right.
\end{equation}

\begin{prop}\label{prop dim=d}
Let $\mu^0,\mu^1\in\mc{P}_c^{ac}(\mb{R}^d)$ 
compactly supported in $\omega$.
Then, for  all  $T>0$, 
System \eqref{eq:transport dim2} is approximately controllable on the time interval $[0,T]$
from $\mu^0$ to $\mu^1$
with a control $u:\mb{R}^d\times\mb{R}^+\rightarrow\mb{R}^d$ uniformly bounded, Lipschitz in space and  measurable in time.
Moreover, the solution $\mu$ to System \eqref{eq:transport dim2} satisfies
 $$\supp(\mu(t))\subset \omega,$$ for all $t\in [0,T]$.
\end{prop}


\begin{proof}[Proof of Proposition \ref{prop dim=d}] We assume that  $d:=2$, but
the reader will see that the proof can be clearly adapted to dimension one or to 
any other space dimension.
In view to simplify the computations, we suppose that $T:=1$ and $\supp(\mu^i)\subset(0,1)^2\subset\subset\omega$ for $i=1,2$. 
 We first partition $(0,1)^2$. 
Let $n\in\mb{N}^*$, consider $a_0:=0$, $b_0:=0$ 
and define the points $a_i,b_i$ for all $i\in\{1,...,n\}$ 
by induction as follows: suppose that for a given $i\in\{0,...,n-1\}$  the points $a_i$ and $b_i$
are defined, then the points  $a_{i+1}$ and $b_{i+1}$ are the smallest values such that 
\begin{equation*}
\displaystyle\int_{(a_i,a_{i+1})\times\mb{R}}d\mu^0 =\frac{1}{n}
\mbox{ ~~~and~~~ }\displaystyle\int_{(b_i,b_{i+1})\times\mb{R}}d\mu^1 =\frac{1}{n}.
\end{equation*}
Again, for each $i\in\{0,...,n-1\}$, we consider $a_{i,0}:=0$, $b_{i,0}:=0$
and supposing that for a given $j\in\{0,...,n-1\}$  
the points $a_{i,j}$ and $b_{i,j}$ are already defined, 
$a_{i,j+1}$ and $b_{i,j+1}$ are the smallest values such that
\begin{equation*}
\displaystyle\int_{A_{ij}}d\mu^0 =\frac{1}{n^2}
\mbox{ ~~~and~~~ }
\displaystyle\int_{B_{ij}}d\mu^1 =\frac{1}{n^2},
\end{equation*}
where $A_{ij}:=(a_i,a_{i+1})\times(a_{ij},a_{i(j+1)})$ 
and $B_{ij}:=(b_i,b_{i+1})\times(b_{ij},b_{i(j+1)})$.
Since $\mu^0$ and $\mu^1$ have a mass equal to $1$ and are supported in $(0,1)^2$, then 
$a_n,b_n\leqslant1$ and 
$a_{i,n},~b_{i,n}\leqslant1$ for all $i\in \{0,...,n-1\}$.
We give in Figure \ref{fig: mesh} an example of such partition.

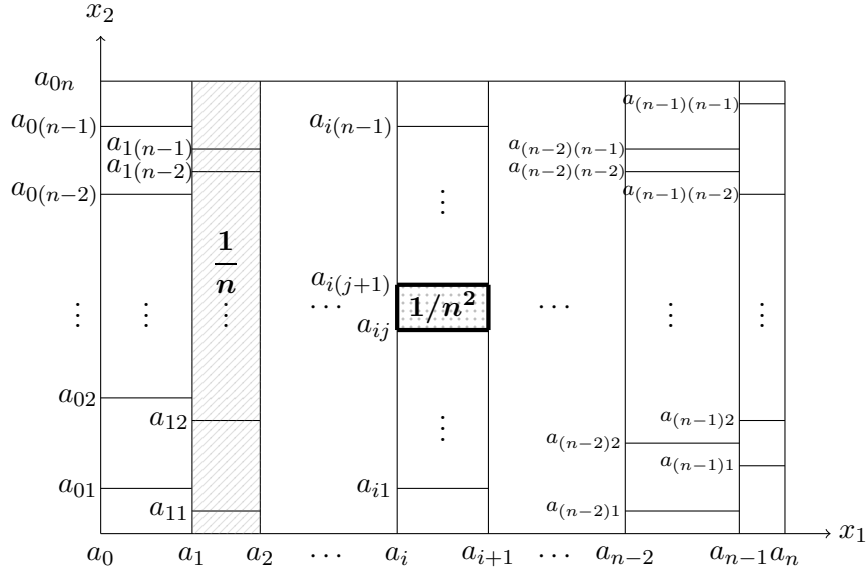
\begin{figure}[ht]
\begin{center}
\begin{tikzpicture}[scale=3]
\draw[->] (0,0) -- (0,2.2);
\draw[->] (0,0) -- (3.2,0);
\draw[-] (0,2) -- (3,2);
\draw[-] (3,2) -- (3,0);
\path (0,2.3) node {$x_2$};
\path (3.3,0) node {$x_1$};
\path (0,-0.1) node {$a_0$};
\draw[-] (0.4,0) -- (0.40,2);
\path (0.4,-0.1) node {$a_1$};
\draw[-] (0,0.2) -- (0.4,0.2);
\path (-0.1,0.2) node {$a_{01}$};
\draw[-] (0,0.6) -- (0.4,0.6);
\path (-0.1,0.6) node {$a_{02}$};
\path (-0.1,1) node {$\vdots$};
\path (0.2,1) node {$\vdots$};
\draw[-] (0,1.5) -- (0.4,1.5);
\path (-0.2,1.5) node {$a_{0(n-2)}$};
\draw[-] (0,1.8) -- (0.4,1.8);
\path (-0.2,1.8) node {$a_{0(n-1)}$};
\path (-0.2,2) node {$a_{0n}$};
\draw[-] (0.7,0) -- (0.7,2);
\path (0.7,-0.1) node {$a_2$};
\draw[-] (0.4,0.1) -- (0.7,0.1);
\path (0.3,0.1) node {$a_{11}$};
\draw[-] (0.4,0.5) -- (0.7,0.5);
\path (0.3,0.5) node {$a_{12}$};
\path (0.55,1) node {$\vdots$};
\draw[-] (0.4,1.6) -- (0.7,1.6);
\path (0.22,1.6) node {$a_{1(n-2)}$};
\draw[-] (0.4,1.7) -- (0.7,1.7);
\path (0.22,1.7) node {$a_{1(n-1)}$};
\fill [opacity=0.5,pattern=north east lines, pattern color = gray] (0.4,0) -- (0.4,2) -- (0.7,2) -- (0.7,0);
\path (0.55,1.2) node {$\boldsymbol{\dfrac{1}{n}}$};

\path (1,-0.1) node {$\cdots$};
\path (1,1) node {$\cdots$};
\draw[-] (1.3,0) -- (1.3,2);
\path (1.3,-0.1) node {$a_i$};

\draw[-] (1.3,0.2) -- (1.7,0.2);
\path (1.2,0.2) node {$a_{i1}$};

\path (1.5,0.5) node {$\vdots$};

\draw[-,ultra thick] (1.3,0.9) -- (1.7,0.9);
\path (1.2,0.9) node {$a_{ij}$};
\draw[-,ultra thick] (1.3,1.1) -- (1.7,1.1);
\path (1.1,1.1) node {$a_{i(j+1)}$};
\path (1.5,1.5) node {$\vdots$};
\draw[-,ultra thick] (1.3,0.9) -- (1.3,1.1);
\draw[-,ultra thick] (1.7,0.9) -- (1.7,1.1);

\fill [opacity=0.5,pattern=dots] (1.3,0.9) -- (1.3,1.1) -- (1.7,1.1) -- (1.7,0.9);
\path (1.5,1.0) node {$\boldsymbol{1/n^2}$};

\draw[-] (1.3,1.8) -- (1.7,1.8);
\path (1.1,1.8) node {$a_{i(n-1)}$};

\draw[-] (1.7,0) -- (1.7,0.9);
\draw[-] (1.7,1.1) -- (1.7,2);
\path (1.7,-0.1) node {$a_{i+1}$};

\path (2,-0.1) node {$\cdots$};
\path (2,1) node {$\cdots$};

\draw[-] (2.3,0) -- (2.3,2);
\path (2.3,-0.1) node {$a_{n-2}$};

\draw[-] (2.3,0.1) -- (2.8,0.1);
\path (2.12,0.1) node {\scriptsize{$a_{(n-2)1}$}};
\draw[-] (2.3,0.4) -- (2.8,0.4);
\path (2.12,0.4) node {\scriptsize{$a_{(n-2)2}$}};
\path (2.5,1) node {$\vdots$};
\draw[-] (2.3,1.6) -- (2.8,1.6);
\path (2.05,1.6) node {\scriptsize{$a_{(n-2)(n-2)}$}};
\draw[-] (2.3,1.7) -- (2.8,1.7);
\path (2.05,1.7) node {\scriptsize{$a_{(n-2)(n-1)}$}};

\draw[-] (2.8,0) -- (2.8,2);
\path (2.8,-0.1) node {$a_{n-1}$};

\draw[-] (2.8,0.3) -- (3,0.3);
\path (2.62,0.3) node {\scriptsize{$a_{(n-1)1}$}};
\draw[-] (2.8,0.5) -- (3,0.5);
\path (2.62,0.5) node {\scriptsize{$a_{ (n-1)2}$}};
\path (2.9,1) node {$\vdots$};
\draw[-] (2.8,1.5) -- (3,1.5);
\path (2.55,1.5) node {\scriptsize{$a_{(n-1)(n-2)}$}};
\draw[-] (2.8,1.9) -- (3,1.9);
\path (2.55,1.9) node {\scriptsize{$a_{(n-1)(n-1)}$}};
\path (3,-0.1) node {$a_n$};

\end{tikzpicture}
\caption{Example of a partition for $\mu^0$.}
\label{fig: mesh}
\end{center}
\end{figure}


If one aims to define a vector field sending each $A_{ij}$ to $B_{ij}$, then some shear stress 
is naturally introduced, as described in  Remark \ref{rmq mesh}.
To overcome this problem, we first define sets $\widetilde{A}_{ij}\subset\subset A_{ij}$  
and  $\widetilde{B}_{ij}\subset\subset B_{ij}$
for all $i,j\in\{0,...,n-1\}$. We then send  the mass of $\mu^0$ from each $\widetilde{A}_{ij}$ to $\widetilde{B}_{ij}$, while we do not control the mass contained in $A_{ij}\backslash\widetilde{A}_{ij}$.
More precisely, for all $i,j\in\{0,...,n-1\}$, we define, as in Figure  \ref{fig:cell}, 
$a_i^-,~a_i^+,a_{ij}^-,~a_{ij}^+$
the smallest values such that
\begin{equation*}
\displaystyle\int_{(a_i,a_{i}^-)\times(a_{ij},a_{i(j+1)})}d\mu^0 =
\displaystyle\int_{(a_i^+,a_{i+1})\times(a_{ij},a_{i(j+1)})}d\mu^0 =\frac{1}{n^3}
\end{equation*}
and
\begin{equation*}
\displaystyle\int_{(a_i^-,a_{i}^+)\times(a_{ij},a_{ij}^-)}d\mu^0 =
\displaystyle\int_{(a_i^-,a_{i}^+)\times(a_{ij}^+,a_{i(j+1)})}d\mu^0 
=\dfrac{1}{n}\times\left(\dfrac{1}{n^2}-\dfrac{2}{n^3}\right).
\end{equation*}
\vspace*{-2mm}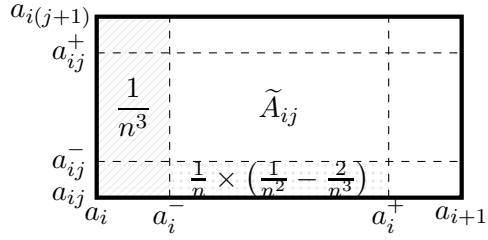
\begin{figure}[ht]
\begin{center}
\begin{tikzpicture}[scale=2.4]
\draw[color=black,ultra thick] (0 ,0) -- (0 ,1) -- (2,1) -- (2 ,0) -- cycle;
\draw[dashed] (0.4,0) -- (0.40,1);
\draw[dashed] (1.6,0) -- (1.6,1);
\draw[dashed] (0,0.2) -- (2,0.2);
\draw[dashed] (0,0.8) -- (2,0.8);
\fill [opacity=0.3,pattern=north east lines] (0,0) -- (0.4,0) -- (0.4,1) -- (0,1);
\fill [opacity=0.3,pattern=dots] (0.4,0) -- (1.6,0) -- (1.6,0.2) -- (0.4,0.2);
\path (1,0.1) node {$\frac{1}{n}\times\left(\frac{1}{n^2}-\frac{2}{n^3}\right)$};
\path (0.2,0.5) node {$\dfrac{1}{n^3}$};
\path (0,-0.1) node {$a_i$};
\path (0.4,-0.1) node {$a_i^-$};
\path (1.6,-0.1) node {$a_i^+$};
\path (2,-0.1) node {$a_{i+1}$};
\path (-0.15,0) node {$a_{ij}$};
\path (-0.15,0.2) node {$a_{ij}^-$};
\path (-0.15,0.8) node {$a_{ij}^+$};
\path (-0.25,1) node {$a_{i(j+1)}$};
\path (1,0.5) node {$\widetilde{A}_{ij}$};
\end{tikzpicture}\vspace*{-4mm}
\caption{Example of  cell.}\label{fig:cell}
\end{center}\end{figure}

\noindent We similarly define $b_i^+,~b_i^-,~b_{ij}^+,~b_{ij}^-$ and finally define
$$\widetilde{A}_{ij}:=(a_i^-,a_i^+)\times(a_{ij}^-,a_{ij}^+)
\mbox{ and }
\widetilde{B}_{ij}:=(b_i^-,b_i^+)\times(b_{ij}^-,b_{ij}^+).$$

The goal is to build a solution to System    \eqref{eq:transport dim2}
such that the corresponding flow $\Phi_t^u$ satisfies  
\vspace*{-2mm}\begin{equation}\label{phi(Aij) in Bij}
\Phi_T^u(\widetilde{A}_{ij})=\widetilde{B}_{ij},
\end{equation}
for all $i,j\in\{0,...,n-1\}$. 
We observe that we do not take into account the displacement of the mass contained in 
$A_{ij}\backslash \widetilde{A}_{ij}$.
We will show  that the mass of  the corresponding term tends to zero when $n$ goes to infinity.
The rest of the proof is divided into two steps. 
In a first step, we build a flow satisfying \eqref{phi(Aij) in Bij}, then the corresponding vector field.
In a second step, we compute the Wasserstein distance between $\mu^1$ and $\mu(T)$,
showing that it converges to zero when $n$ goes to infinity.

 {\bf Step 1:} 
 We first build a flow satisfying \eqref{phi(Aij) in Bij}.
We recall that $T:=1$.
For each $i\in\{0,...,n-1\}$, we denote by $c^-_i$ and $c^+_i$
  the linear functions equal to $a_i^-$ and $a_i^+$ at time $t=0$ and equal to $b_i^-$
  and $b_i^+$ at time $t=T=1$, respectively,
\textit{i.e.} the functions defined for all $t\in[0,T]$ by:
\begin{equation*}
c^-_i(t)=(b_i^--a_i^-)t+a_i^-
\mbox{~~ and~ ~}
c^+_i(t)=(b_i^+-a_i^+)t+a_i^+.
\end{equation*}
Similarly, for all $i,j\in\{0,...,n-1\}$,  we denote by $c^-_{ij}$  and $c^+_{ij}$
the linear functions equal to $a_{ij}^-$ and $a_{ij}^+$ at time $t=0$ and equal to 
$b_{ij}^-$ and $b_{ij}^+$ at time $t=T=1$, respectively, \textit{i.e.} the functions defined for all $t\in[0,T]$ by:
\begin{equation*}
c^-_{ij}(t)=(b_{ij}^--a_{ij}^-)t+a_{ij}^-
\mbox{~~ and~ ~}
c^+_{ij}(t)=(b_{ij}^+-a_{ij}^+)t+a_{ij}^+.
\end{equation*}

Consider the application being the following  linear combination of $c_i^-,~c_i^+$ 
 and $c_{ij}^-,~c_{ij}^+$ on $\widetilde{A}_{ij}$, \textit{i.e.}
\begin{equation}\label{expr char 2d}
x(x^0,t):=
\left(\begin{array}{c}
x_1(x^0,t)\\
x_2(x^0,t)\end{array}\right)
=\left(\begin{array}{c}
\dfrac{a_i^+-x^0_1}{a_i^+-a_i^-}c^-_i(t)
+\dfrac{x^0_1-a_i^-}{a_i^+-a_i^-}c^+_i(t)\\\noalign{\smallskip}
\dfrac{a_{ij}^+-x^0_2}{a_{ij}^+-a_{ij}^-}c^-_{ij}(t)
+\dfrac{x^0_2-a_{ij}^-}{a_{ij}^+-a_{ij}^-}c^+_{ij}(t)
\end{array}\right),
\end{equation}
where $x^0=(x^0_1,x^0_2)\in \widetilde{A}_{ij}$.
Let us prove that an extension of the application $(x^0,t)\mapsto x(x^0,t)$ is a flow associated to a vector field $u$.
After some computations, we obtain
 \begin{equation*}
\left\{\begin{array}{ll}
\dfrac{dx_{1}}{dt}(x^0,t)=\alpha_{i}(t)x_{1}(x^0,t)+\beta_i(t)&~\forall t\in[0,T],\\\noalign{\smallskip}
\dfrac{dx_{2}}{dt}(x^0,t)=\alpha_{ij}(t)x_{2}(x^0,t)+\beta_{ij}(t)&~\forall t\in[0,T],
\end{array}\right.
\end{equation*}
 where for all $t\in[0,T]$,
  \begin{equation*}
\left\{\begin{array}{l}
\alpha_i(t)=\dfrac{b_i^+-b_i^-+a_i^--a_i^+}{c^+_{i}(t)-c^-_{i}(t)},
~~\beta_i(t)=\dfrac{a_i^+b_{i}^--a_i^-b_i^+}{c_i^+(t)-c_i^-(t)},\\\noalign{\smallskip}
\alpha_{ij}(t)=\dfrac{b_{ij}^+-b_{ij}^-+a_{ij}^--a_{ij}^+}{c^+_{ij}(t)-c^-_{ij}(t)},
~~\beta_{ij}(t)=\dfrac{a_{ij}^+b_{ij}^--a_{ij}^-b_{ij}^+}{c_{ij}^+(t)-c_{ij}^-(t)}.
\end{array}\right.
\end{equation*}
The last quantities are well defined since 
for all $i,j\in\{0,...,n-1\}$ and $t\in[0,T]$
 \begin{equation*}
\left\{\begin{array}{l}
|c^+_{i}(t)-c^-_{i}(t)|\geqslant \max\{|a_{i}^+-a_{i}^-|,|b_{i}^+-b_{i}^-|\},\\
|c_{ij}^+(t)-c_{ij}^-(t)|\geqslant \max\{|a_{ij}^+-a_{ij}^-|,|b_{ij}^+-b_{ij}^-|\}.
\end{array}\right.
\end{equation*}
For all $t\in [0,T]$, consider the set 
\begin{equation*}
\widetilde{C}_{ij}(t):=(c_i^-(t),c_i^+(t))\times (c_{ij}^-(t),c_{ij}^+(t)).
\end{equation*}
We remark that $\widetilde{C}_{ij}(0)=\widetilde{A}_{ij}$ and 
$\widetilde{C}_{ij}(T)=\widetilde{B}_{ij}$.
On $$\widetilde{C}_{ij}:=\{(x,t): t\in [0,T], x\in \widetilde{C}_{ij}(t)\},$$ we then define the vector field $u$ by
  \begin{equation*}
\left\{\begin{array}{l}
u_1(x,t)=\alpha_{i}(t)x_1+\beta_i(t),\\
u_2(x,t)=\alpha_{ij}(t)x_2+\beta_{ij}(t),
\end{array}\right.
\end{equation*}
 for all $(x,t)\in \widetilde{C}_{ij}$ ($x=(x_1,x_2)$). 
 We extend $u$ by a uniform bounded  $\mc{C}^{\infty}$ function outside  $\cup_{ij}\widetilde{C}_{ij}$,
 then $u$ is 
 a $\mc{C}^{\infty}$ function and it satisfies $\supp(u)\subset \omega$.
Then, System \eqref{eq:transport} admits an unique solution and the 
flow on $\widetilde{C}_{ij}$ is given by  \eqref{expr char 2d}.

 {\bf Step 2:} We now prove that the refinement of the grid provides convergence to the target $\mu^1$, \textit{i.e.}
\begin{equation*}
W_1(\mu^1,\mu(T))\underset{n\rightarrow\infty}{\longrightarrow}0.
\end{equation*}
We remark that 
 \begin{equation*}
 \int_{\widetilde{B}_{ij}}d\mu(T)
 =\int_{\widetilde{B}_{ij}}d\mu^1
 = \frac{1}{n^2}-\dfrac{2}{n^3}-\dfrac{2}{n}\left(\dfrac{1}{n^2} -\dfrac{2}{n^3}\right)
 =\dfrac{(n-2)^2}{n^4}.
 \end{equation*}
Hence, by defining 
$$ R:=(0,1)^2~\backslash~ \bigcup\limits_{ij}\widetilde{B}_{ij},$$
 we also have
 \begin{equation*}
 \int_{R}d\mu(T) =\int_{R}d\mu^1
 =1-\dfrac{(n-2)^2}{n^2}.
 \end{equation*}
 Using \eqref{ine wasser},
it holds
 \begin{equation}\label{ine norm 1}
\begin{array}{rcl}
W_1(\mu^1,\mu(T))
&\leqslant&\sum\limits_{i,j=1}^nW_1(\mu^1_{|\widetilde{B}_{ij}},\mu(T)_{|\widetilde{B}_{ij}})+
W_1(\mu^1_{|R},\mu(T)_{|R}).
\end{array}
\end{equation}
We now estimate each term in the right-hand side of  \eqref{ine norm 1}.
 Since we deal with AC measures, 
 using Properties \ref{prop Wp}, 
 there exist measurable maps  
  $\gamma_{ij}:\mb{R}^2\rightarrow\mb{R}^2$, for all $i,j\in\{0,...,n-1\}$, 
  and   $\overline{\gamma} :\mb{R}^2\rightarrow\mb{R}^2$ such that
 \begin{equation*}
\left\{\begin{array}{l}
 \gamma_{ij}\#(\mu^1_{|\widetilde{B}_{ij}})
  =\mu(T)_{|\widetilde{B}_{ij}},\\\noalign{\smallskip}
  W_1(\mu^1_{|\widetilde{B}_{ij}},\mu(T)_{|\widetilde{B}_{ij}})\\\noalign{\smallskip}\hspace*{1cm}
=\displaystyle\int_{\widetilde{B}_{ij}}|x-\gamma_{ij}(x)|d\mu^1(x)
\end{array}\right.
\qquad\mbox{and}\qquad
\left\{\begin{array}{l}
 \overline{\gamma}\#(\mu^1_{|R})
 =\mu(T)_{|R},\\\noalign{\smallskip}
  W_1(\mu^1_{|R},\mu(T)_{|R})\\\noalign{\smallskip}\hspace*{1cm}
=\displaystyle\int_{R}|x-\overline{\gamma}(x)|d\mu^1(x).
 \end{array}\right.
 \end{equation*}
In the first term in the right hand side of  \eqref{ine norm 1}, observe that $\gamma_{ij}$ moves masses inside $\widetilde{B}_{ij}$ only.
Thus, for all $i,j\in\{0,..., n-1\}$,
 using the triangle inequality,
 \begin{equation}\label{ine norm 2}
\begin{array}{c}
W_1(\mu^1_{|\widetilde{B}_{ij}},\mu(T)_{|\widetilde{B}_{ij}})
= \displaystyle\int_{\widetilde{B}_{ij}}|x-\gamma_{ij}(x)|d\mu^1(x)\\
\leqslant [(b_i^+-b_i^-)+(b_{ij}^+-b_{ij}^-)]
\displaystyle\int_{\widetilde{B}_{ij}}d\mu^1(x)
\leqslant (b_i^+-b_i^-+b_{ij}^+-b_{ij}^-)
\dfrac{(n-2)^2}{n^4}.
\end{array}
 \end{equation}
For the second term in the right-hand side of \eqref{ine norm 1}, 
observe that $\overline{\gamma}$ moves a small mass in the bounded set $(0,1)$.
Thus it holds
 \begin{equation}\label{ine norm 3}
 \begin{array}{rcl}
 W_1(\mu^1_{|R},\mu(T)_{|R})
&=&\displaystyle\int_{R}|x-\overline{\gamma}(x)|d\mu^1(x)
\leqslant 2\ \left(1-\dfrac{(n-2)^2}{n^2}\right)
=8\dfrac{n-1}{n^2}.
\end{array}
 \end{equation}
Combining \eqref{ine norm 1}, \eqref{ine norm 2} and \eqref{ine norm 3}, we obtain
\begin{equation*}
\begin{array}{rcl}
W_1(\mu^1,\mu(T))
&\leqslant&
\left(\sum\limits_{i,j=1}^n(b_i^+-b_i^-+b_{ij}^+-b_{ij}^-)\dfrac{(n-2)^2}{n^4}\right) +8\dfrac{n-1}{n^2}\\
&\leqslant& 2n\dfrac{(n-2)^2}{n^4}+ 8\dfrac{n-1}{n^2}
\underset{n\rightarrow\infty}{\longrightarrow}0.
\end{array}
\end{equation*}
$\left.\right.$
\end{proof}

\begin{rmq}\label{rmq mesh}
It is not possible in general to build a Lipschitz vector field sending directly 
each $A_{ij}$ to $B_{ij}$ using the strategy 
developed in the proof of Proposition \ref{prop dim=d}.
Indeed, we would obtain discontinuous velocities on the lines $c_i$.
Figure \ref{fig: shear stress} illustrates this phenomenon in the case $n=2$.

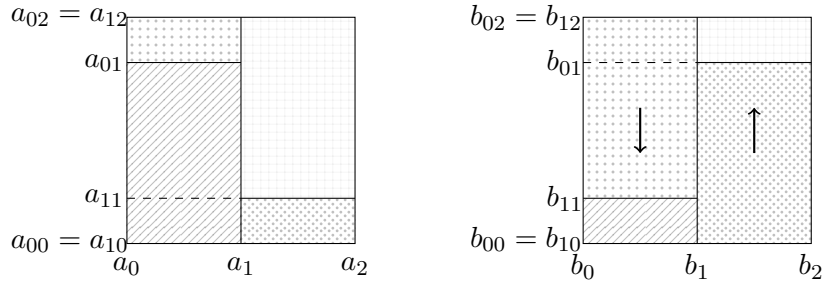
\begin{figure}[ht]
\begin{center}
\begin{tikzpicture}[scale=3]
\fill [opacity=0.5,pattern=north east lines] (0,0) -- (0.5,0) -- (0.5,0.8) -- (0,0.8);
\fill [opacity=0.5,pattern=dots] (0,0.8) -- (0.5,0.8) -- (0.5,1) -- (0,1);
\fill [opacity=0.5,pattern=crosshatch dots] (0.5,0) -- (1,0) -- (1,0.2) -- (0.5,0.2);
\fill [opacity=0.2,pattern=grid] (0.5,0.2) -- (1,0.2) -- (1,1) -- (0.5,1);

\fill [opacity=0.5,pattern=north east lines] (2,0) -- (2.5,0) -- (2.5,0.2) -- (2,0.2);
\fill [opacity=0.5,pattern=dots] (2,0.2) -- (2.5,0.2) -- (2.5,1) -- (2,1);
\fill [opacity=0.5,pattern=crosshatch dots] (2.5,0) -- (3,0) -- (3,0.8) -- (2.5,0.8);
\fill [opacity=0.2,pattern=grid] (2.5,0.8) -- (3,0.8) -- (3,1) -- (2.5,1);

\draw[color=black] (0 ,0) -- (0 ,1) -- (1,1) -- (1 ,0) -- cycle;
\draw[color=black] (2 ,0) -- (2 ,1) -- (3,1) -- (3 ,0) -- cycle;
\draw[-] (0.5,0) -- (0.5,1);
\draw[-] (0,0.8) -- (0.5,0.8);
\draw[-] (0.5,0.2) -- (1,0.2);
\draw[dashed] (0,0.2) -- (0.5,0.2);
\path (0,-0.1) node {$a_0$};
\path (0.5,-0.1) node {$a_1$};
\path (1,-0.1) node {$a_2$};
\path (-0.25,0) node {$a_{00}=a_{10}$};
\path (-0.1,0.2) node {$a_{11}$};
\path (-0.1,0.8) node {$a_{01}$};
\path (-0.25,1) node {$a_{02}=a_{12}$};

\draw[->,thick] (2.25,0.6) -- (2.25,0.4);
\draw[->,thick] (2.75,0.4) -- (2.75,0.6);

\draw[-] (2.5,0) -- (2.5,1);
\draw[-] (2,0.2) -- (2.5,0.2);
\draw[-] (2.5,0.8) -- (3,0.8);
\draw[dashed] (2,0.8) -- (2.5,0.8);
\path (2,-0.1) node {$b_0$};
\path (2.5,-0.1) node {$b_1$};
\path (3,-0.1) node {$b_2$};
\path (1.75,0.02) node {$b_{00}=b_{10}$};
\path (1.92,0.2) node {$b_{11}$};
\path (1.92,0.8) node {$b_{01}$};
\path (1.75,1) node {$b_{02}=b_{12}$};
\end{tikzpicture}
\caption{Shear stress  (left: $\mu^0$, right:  $\mu^1$)}
\label{fig: shear stress}
\end{center}\end{figure}
\end{rmq}

%

\subsection{Approximate controllability with a localized regular control}\label{sec:approx loc}

This section is devoted to prove Theorem \ref{th cont approx}: we aim to prove approximate controllability of System \eqref{eq:transport} with a Lipschitz localized control.
This means that we remove the constraints $\supp(\mu^0)\subset \omega$, 
$\supp(\mu^1)\subset \omega$ and 
$v:= 0$, that we used in Section 
\ref{sec:approx dimn}. 
On the other side, we impose Condition \ref{cond1}.

Before the main proof, we need three useful results. First of all, we give a consequence of Condition \ref{cond1}:

\begin{cond}\label{cond2}
There exist two real numbers $T_0^*$, $T_1^*>0$ and a nonempty open set $\omega_0\subset\subset\omega$ such that 
\begin{enumerate}
\item[(i)] For each $x^0\in \supp(\mu^0)$, 
there exists $t^0\in[0, T_0^*]$ such that $\Phi_{t^0}^v(x^0)\in \omega_0,$
where $\Phi_{t}^v$ is the flow associated to $v$.

\item[(ii)] For each $x^1\in \supp(\mu^1)$, 
there exists $t^1\in[0,T_1^*]$ such that $\Phi_{-t^1}^{v}(x^1)\in \omega_0$.
\end{enumerate}
\end{cond}

\begin{lemma}\label{lemma cond}
If Condition \ref{cond1}  is satisfied for $\mu^0,~\mu^1\in \mc{P}_c(\mb{R}^d)$, then  Condition \ref{cond2} is satisfied too.
\end{lemma}

\begin{proof}
We use a compactness argument. Let $\mu^0\in\mc{P}_c(\mb{R}^d)$
and assume that Condition  \ref{cond1} holds.
Let  $x^0\in\supp(\mu^0)$.
Using Condition  \ref{cond1},
 there exists $t^0(x^0)>0$ such that $\Phi_{t^0(x^0)}^v(x^0)\in \omega.$
 Choose $r(x^0)>0$ such that $B_{r(x^0)}(\Phi_{t^0(x^0)}^v(x^0))\subset\subset \omega$, where $B_r(x^0)$ denotes the open ball of radius $r > 0$ centered at point $x^0$  in $\mb{R}^d$. 
 Such $r(x^0)$ exists, since $\omega$ is open. By continuity of the application $x^1\mapsto \Phi_{t^0(x^0)}^v(x^1) $ (see  \cite[Th. 2.1.1]{BP07}), 
 there exists $\hat{r}(x^0)$ such that 
 $$x^1\in B_{\hat{r}(x^0)}(x^0)~~\Rightarrow~~\Phi_{t^0(x^0)}^v(x^1) \in B_{r(x^0)}(\Phi_{t^0(x^0)}^v(x^0)).$$
Since $\mu^0$ is compactly supported,  we can find a set $\{x^0_1,...,x^0_{N_0}\}\subset\supp(\mu^0)$ such that 
$$\supp(\mu^0)\subset \bigcup\limits_{i=1}^{N_0}B_{\hat{r}(x^0_i)}(x_i^0).$$
We similarly build a set $\{x^1_1,...,x^1_{N_1}\}\subset\supp(\mu^1)$.
Thus  Condition \ref{cond2} is satisfied for 
$$T_k^*:=\max\{ t^k(x^k_i):i\in\{1,...,{N_k}\}\},$$
with $k=0,1$ and 
$$\omega_0:=\left(\bigcup\limits_{i=1}^{N_0}B_{r(x^0_i)}(\Phi^v_{t^0(x^0_i)}(x^0_i))\right)
\bigcup\left(\bigcup\limits_{i=1}^{N_1}B_{r(x^1_i)}(\Phi^v_{t^1(x^1_i)}(x^1_i))\right)\subset\subset\omega
.$$
$\left.\right.$
\end{proof}
The second useful result is the following proposition, showing that we can store a large part of 
the mass of $\mu^0$ in $\omega$, under Condition \ref{cond2}.

\begin{prop}\label{prop1}
Let $\mu^0\in\mc{P}_c^{ac}(\mb{R}^d)$ satisfying the first item of Condition \ref{cond2}.
Then, for all $\varepsilon>0$, 
there exists a space-dependent vector field $\mathds{1}_{\omega}u$ 
Lipschitz and uniformly bound\-ed
and a Borel set $A\subset\mb{R}^d$
such that 
\begin{equation}\label{mu T0}
\mu^0(A)=\varepsilon\mbox{ and }
\supp(\Phi^{v+\mathds{1}_{\omega}u}_{T_0^*}\#\mu^0_{|A^c})\subset\omega.
\end{equation}
\end{prop}

\begin{proof}
For each $k\in\mb{N}^*$, we  denote by $\omega_k$ the closed set 
 defined by
\begin{equation*}
\omega_k:=\{x^0\in\mb{R}^d:d(x^0,\omega^c_0)\geqslant 1/k\}
\end{equation*}
and a cutoff function 
$\theta_k\in \mc{C}^{\infty}(\mb{R}^d)$ satisfying 
\begin{equation*}
\left\{\begin{array}{l}
0\leqslant \theta_k\leqslant 1,\\
\theta_k=1\mbox{ in }\omega_0^c,\\
\theta_k=0\mbox{ in }\omega_k.
\end{array}\right.
\end{equation*}
For all $x^0\in\supp(\mu^0)$, we define
$$t_0(x^0):=\inf\{t\in\mb{R}^+:\Phi_t^v(x^0)\in \omega_0\}
\mbox{ and }
t_k(x^0):=\inf\{t\in\overline{\mb{R}}^+:\Phi_t^v(x^0)\in \omega_k\}.$$
For all $k\in\mb{N}^*$, we consider 
\begin{equation}\label{eq:control prop 35}
u_k:=(\theta_k-1) v
\end{equation}
and
$$S_k:=\{x^0\in \supp(\mu^0)\backslash\omega_0:
\exists s\in (t_0(x^0),t_k(x^0)), \mbox{~s.t.~}\Phi_s^v(x^0)\in \overline{\omega}_0^c\}.$$
The rest of the proof is divided into three steps:
\begin{itemize}
\item In Step 1, we prove that the range of the flow associated to $x^0$ with the control $u_k$ is included in the range of the flow associated to $x^0$ without control.
\item In Step 2, we show that  $S_k$ is a Borel set for all $k\in\mb{N}^*$.
\item In  Step 3, we prove that for a $K$ large enough we have 
 \begin{equation}\label{eq Sk eps}
\mu^0(\omega\backslash\omega_K)+\mu^0(S_K)\leqslant\varepsilon.
\end{equation}
\end{itemize}

\textbf{Step 1:} 
Consider the flow $y(t):=\Phi_{t}^{v}(x^0)$ associated to $x^0$ without control, \textit{i.e.}
 the solution to 
\begin{equation*}
\left\{\begin{array}{l}
\dot{y}(t)=v(y(t)),~t\geqslant 0,\\\noalign{\smallskip}
y(0)=x^0
\end{array}\right.
\end{equation*}
and the flow  $z_k(t):=\Phi_{t}^{v+u_k}(x^0)$ associated to $x^0$ with the control $u_k$ 
given in \eqref{eq:control prop 35}, \textit{i.e.} the solution to
\begin{equation}\label{eq:carac y}
\left\{\begin{array}{l}
\dot{z}_k(t)=(v+u_k)(z_k(t))=\theta_k(z_k(t))\times v(z_k(t)),~t\geqslant 0,\\\noalign{\smallskip}
z_k(0)=x^0.
\end{array}\right.
\end{equation}
Consider the solution $\gamma_k$ to the following system 
\begin{equation}\label{eq:carac gamma}
\left\{\begin{array}{l}
\dot{\gamma_k}(t)=\theta_k(y(\gamma_k(t))),~t\geqslant 0,\\\noalign{\smallskip}
\gamma(0)=0.
\end{array}\right.
\end{equation}
Since $\theta_k$ and $y$ are Lipschitz, then System \eqref{eq:carac gamma} admits a solution defined for all times.
We remark that $\xi_k:=y\circ \gamma_k$ is solution to System \eqref{eq:carac y}. 
Indeed, for all $t\geqslant 0$ it holds
\begin{equation*}
\left\{\begin{array}{l}
\dot\xi_k(t)=\dot\gamma_k(t)\times\dot y(\gamma_k(t))
=\theta_k(\xi_k(t))\times v(\xi_k(t)),~t\geqslant 0,\\\noalign{\smallskip}
\xi_k(0)=y(\gamma_k(0))=y(0).
\end{array}\right.
\end{equation*}
By uniqueness of the solution to System \eqref{eq:carac y}, we obtain
\begin{equation*}
y(\gamma_k(t))=z_k(t)\mbox{ for all }t\geqslant 0.
\end{equation*}
Using the fact that $0\leqslant \theta\leqslant 1$ and the definition of $\gamma_k$, we have
 $$\left\{\begin{array}{ll}
 \gamma_k\mbox{ increasing},&\\
 \gamma_k(t)\leqslant t&~\forall t\in[0,t_k(x^0)],\\
 \gamma_k(t)\leqslant t_k(x^0)&~\forall t\geqslant t_k(x^0).
 \end{array}\right. $$
We deduce  that, for all $x^0\in\supp(\mu^0)$, it holds
$$\{z_k(t):t\geqslant 0\}
 \subset\{y(s):s\in[0,t_k(x^0)]\}.$$

\textbf{Step 2:} 
We now prove that $S_k$ is a Borel set 
by showing that 
the set 
$$R_k:=\{x^0\in \mb{R}^d:t_0(x^0)<\infty\mbox{ and }
\exists s\in (t_0(x^0),t_k(x^0))\mbox{ s.t. } \Phi_s^v(x^0)\in \overline{\omega}_0^c\}$$
is open.
Let $k\in\mb{N}^*$,  $x^0$ be an element of $R_k$
and search $r(x^0)>0$ such that $B_{r(x^0)}(x^0)\subset R_k$.  
There exists $s\in (t_0(x^0),t_k(x^0))$ such that $\Phi_s^v(x^0)\in \overline{\omega}_0^c$.
Since  $\overline{\omega}^c_0$ is open, 
for a $\beta>0$, we have
$B_{\beta}(\Phi_{s}^{v}(x^0))\subset \overline{\omega}_0^c$.
By continuity of the application $x^1\mapsto \Phi_{s}^v(x^1)$,
there exists $r(x^0)>0$ such that 
$$x^1\in B_{r(x^0)}(x^0)\Rightarrow \Phi_{s}^v(x^1)\in B_{\beta}(\Phi_{s}^{v}(x^0)).$$
%
Thus, for all $k\in\mb{N}^*$, $R_k$ is open and $S_k$ is a Borel set.

\textbf{Step 3:}
We now prove that \eqref{eq Sk eps} holds for a $K$ large enough.
Since we deal we AC measure,
there exists $K_0\in\mb{N}^*$ such that  for all $k\geqslant K_0$
$$\mu^0(\omega_0\backslash \omega_{k})\leqslant \varepsilon/2.$$ 
Argue now by contradiction to prove 
that  there exists  $K_1\geqslant K_0$ such that 
\begin{equation*}
\mu^0(S_{K_1})\leqslant\varepsilon/2.
\end{equation*}
Assume that   $\mu^0(S_k)>\varepsilon/2$ for all $k\geqslant K_0$.
Using the inclusion $S_{k+1}\subset S_k$, we deduce that 
$$\mu^0\left(\bigcap_{k\in\mb{N}^*}S_k\right)\geqslant\varepsilon/2.$$
Since $\mu^0$ is absolute continuous with respect to $\lambda$ (the Lebesgue measure), 
there exists $\alpha>0$ such that
$$\lambda\left(\bigcap_{k\in\mb{N}^*}S_k\right)\geqslant\alpha.$$
We deduce that the intersection of the set $S_k$ is nonempty. 
Let $\overline{x}^0\in\supp(\mu^0)\backslash\overline{\omega}_0$ be an element of this intersection. 
By definition of $S_k$, for all $k\geqslant K_0$, there exists $s_k$ satisfying 
\begin{equation}\label{conv sk}
\left\{\begin{array}{l}
s_k\in (t_0(\overline{x}^0),t_k(\overline{x}^0)),\\
\Phi_{s_k}^{ v}(\overline{x}^0)\in \overline{\omega}_0^c.
\end{array}\right.
\end{equation}
Moreover, the convergence of $t_k(\overline{x}^0)$ to $t_0(\overline{x}^0)$, implies that 
\begin{equation}\label{eq sk conv}
s_k\rightarrow t_0(\overline{x}^0).
\end{equation}
Using the continuity of $x^1\mapsto\Phi^{v}_t(x^1)$  and the definition of $t_0(x^0)$, 
there exists $\beta>0$ such that 
\begin{equation}\label{eq x(t) in }
\Phi^{v}_t(\overline{x}^0)\in \omega_0 \mbox{ for all }t\in(t_0,t_0+\beta).
\end{equation}
We deduce that \eqref{eq x(t) in } contradicts \eqref{conv sk} and \eqref{eq sk conv}.
Thus  there exists $K\in \mb{N}^*$ such that 
$$\mu^0(S_K)+\mu^0(\omega\backslash \omega_K)\leqslant \varepsilon.$$
Since we deal with AC measures, we add a Borel set to have the equality in \eqref{mu T0}, \textit{i.e.}
 there exists a Borel set $S$ 
 such that 
 $$\mu^0(S_K\cup \omega\backslash \omega_K \cup S)=\varepsilon.$$
We conclude that, for $u$ defined by
\begin{equation*}
u(t):=u^1:=u_K \mbox{ for all }t\in [0,T^*_0],
\end{equation*}
and $A:=S_K\cup\omega\backslash \omega_K\cup S$,
Properties \eqref{mu T0} are satisfied.
\end{proof}

The third useful result for the proof of Theorem \ref{th cont approx} allows 
to approximately steer a measure contained in $\omega$ to a measure contained in an open hypercube $S\subset\subset \omega$.

\begin{prop}\label{prop2}
Let $\mu^0\in\mc{P}_c^{ac}(\mb{R}^d)$ satisfying  $\supp(\mu^0)\subset\omega.$
Define an open hypercube $S$ strictly included in $\omega\backslash\supp(\mu^0)$ and choose $\delta>0$. 
Then, for all $\varepsilon>0$, there exists a vector field $\mathds{1}_{\omega}u$, Lipschitz  and uniformly bounded
 and a Borel set $A$ such that
$$\mu^0(A)=\varepsilon\mbox{ and }
\supp(\Phi^{v+\mathds{1}_{\omega}u}_{\delta}\#\mu^0_{|A^c})\subset S.$$
\end{prop}

\begin{proof}
Consider $S_0$ a nonempty open set of $\mb{R}^d$ of class $\mc{C}^{\infty}$ strictly included in $S$ and $\omega_1$ an open set of $\mb{R}^d$ of class $\mc{C}^{\infty}$ satisfying 
$$\supp(\mu^0)\cup S\subset\subset\omega_1\subset\subset\omega.$$
An example is given in Figure \ref{fig:constr S}.
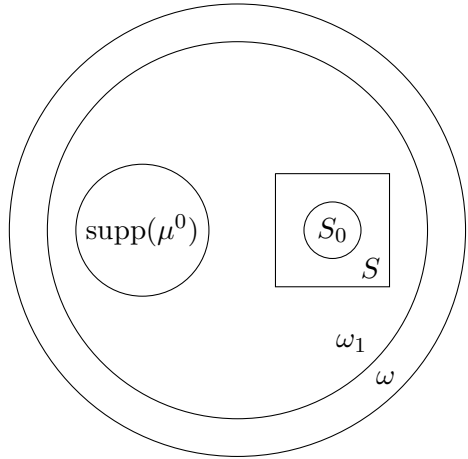
\begin{figure}[h]
\begin{center}
\begin{tikzpicture}[scale=2.5]
\draw (1,1) circle (1cm);
\draw (1,1) circle (1.2cm);
\draw (0.5,1) circle (0.35cm);
\draw (1.2,1.3) -- (1.2,0.7) -- (1.8,0.7) -- (1.8,1.3) -- (1.2,1.3);
\draw (1.5,1) circle (0.15cm);
\path (1.78,0.22) node {$\omega$};
\path (1.6,0.4) node {$\omega_1$};
\path (1.7,0.8) node {$S$};
\path (1.5,1) node {$S_0$};
\path (0.5,1) node {$\supp(\mu^0)$};
\end{tikzpicture}
\caption{Construction of $\omega_1$}\label{fig:constr S}
\end{center}\end{figure}
From \cite[Lemma 1.1, Chap. 1]{FI96} (see also \cite[Lemma 2.68, Chap. 2]{C09}), there exists 
 a function $\eta \in\mathcal{C}^2(\overline{\omega_1})$  satisfying 
\begin{equation}\label{prop eta}
 \kappa_0\leqslant|\nabla\eta |\leqslant \kappa_1\mathrm{~in~}\omega_1\backslash S_0,~~~
  \eta>0\mathrm{~in~}\omega_1 ~~~\mathrm{and}~~~
   \eta=0\mathrm{~on~}\partial\omega_1,
\end{equation}
with $\kappa_0,\kappa_1>0$. 
Let $k\in\mb{N}^*$.
Consider $u_k:\mb{R}^d\rightarrow\mb{R}^d$ Lipschitz and uniformly bounded satisfying
\vspace*{-3mm}\begin{equation*}
u_k:=\left\{\begin{array}{ll}
k\nabla \eta -v&\mbox{in}~ \omega_1,\\
0&\mbox{in}~\omega^c.
\end{array}\right.
\vspace*{-2mm}\end{equation*}

Let $x^0\in \supp(\mu^0)$. 
Consider the flow $z_k(t)=\Phi_t^{v+u_k}(x^0)$ associated to $x^0$ with the control $u_k$,
\textit{i.e.} the solution to system 
\vspace*{-2mm}\begin{equation}\label{eq:carac y bis}
\left\{\begin{array}{l}
\dot{z}_k(t)=v(z_k(t))+u_k(z_k(t)),~t\geqslant 0,\\\noalign{\smallskip}
z_k(0)=x^0.
\end{array}\right.
\vspace*{-2mm}\end{equation}
The different conditions in \eqref{prop eta} imply that 
\vspace*{-2mm}\begin{equation}\label{normal}
n\cdot\nabla\eta<C<0\mbox{ on }  \partial\omega_1, 
\vspace*{-2mm}\end{equation}
where $n$ represents the outward unit normal to $\partial\omega_1$.
Since $\supp(\mu^0)\subset\omega_1$, it holds $z_k(t)\in\omega_1$ for all $t\geqslant 0$,
otherwise, by taking the scalar product of \eqref{eq:carac y bis} and $n$ on $\partial\omega_1$, we obtain a contradiction
with \eqref{normal}.

We now prove that there exists $K(x^0)\in\mb{N}^*$ such that for all $k\geqslant K(x^0)$ 
there exists $t_k(x^0)\in(0,\delta)$ such that 
$z_k(t_k(x^0))$ belongs to $S_0$. By contradiction, assume that there exists a sequences 
$\{k_n\}_{n\in\mb{N}^*}\subset \mb{N}^*$
such that for all $t\in(0,\delta)$
\vspace*{-1mm}\begin{equation}\label{eq:Phi not in omega}
z_{k_n}(t)\in S_0^c.
\vspace*{-1mm}\end{equation}
Consider the function $f_n$ defined  for all $t\in[0,\delta]$ by
\vspace*{-1mm}\begin{equation}\label{def fn}
f_n(t):=k_n\eta(z_{k_n}(t)).
\vspace*{-1mm}\end{equation}
Its time derivative is given for all $t\in[0,\delta]$ by
\vspace*{-1mm}\begin{equation*}
\dot f_n(t)=k_n\dot z_{k_n}(t) \cdot \nabla \eta (z_{k_n}(t))=k_n^2|\nabla \eta(z_{k_n}(t))|^2
\vspace*{-1mm}\end{equation*}
Then, using \eqref{eq:Phi not in omega}, 
 properties \eqref{prop eta} of $\eta$ and definition \eqref{def fn} of $f_n$, it holds
\vspace*{-1mm}\begin{equation*} 
f_n(\delta)\geqslant
k_n^2\kappa_0^2\delta
\mbox{ ~~and ~~}
f_n(\delta)\leqslant k_n\|\eta\|_{\infty}.
\vspace*{-1mm}\end{equation*}
We observe that 
the two last inequalities are in contradiction for $n$ large enough.
Then there exists $K(x^0)\in\mb{N}^*$ such that for all $k\geqslant K(x^0)$ there exists $t_k(x^0)\in(0,\delta)$ such that $z_k(t_k(x^0))$ belongs to $S_0$.

By continuity, there exists $r(x^0)>0$  such that $\Phi_{t_{K(x^0)}(x^0)}^{v+u_{K(x^0)}}(x^1) $ belongs to $S_0$ for all $x^1\in B_{r(x^0)}(x^0)$. 
Since $v+u_k$ is linear with respect to $k$ in $\omega_1$,  then, using the same argument as in Step 1 of the proof of Proposition \ref{prop1}, 
the range of the flow $\Phi_{\cdot}^{v+u_k}$ is independent of $k$.
Thus, for all $k\geqslant K(x^0)$ there exists $t_k^0(x^0)\in(0,\delta)$ such that $\Phi_{t^0_k(x^0)}^{v+u_{k}}(x^1) \in S_0$ for all $x^1\in B_{r(x^0)}(x^0)$. 

By compactness, there exists $\{x^0_1,...,x^0_{N_0}\}$ such that 
\vspace*{-2mm}\begin{equation*}
\supp(\mu^0)\subset \bigcup_{i=1}^{N_0}B_{r(x^0_i)}(x^0_i).
\vspace*{-2mm}\end{equation*}
We deduce that for
$K:=\max_i \{K(x^0_i)\}$,
 for all $x^0\in \supp(\mu^0)$
there exists $t^0(x^0)$ for which $\Phi_{t^0(x^0)}^{v+u_K}(x^0)$ belongs to $S_0$. 
We remark that the first item of Condition  \ref{cond2}  holds replacing $\omega$, $\omega_0$ and $T_0^*$ by $S$, $S_0$ and $\delta$, respectively. 
We conclude applying Proposition \ref{prop1} replacing $\omega$, $\omega_0$, $T_0^*$ and $v$ by $S$, $S_0$, $\delta$ and $v+u_K$, respectively.
\end{proof}

\begin{rmq}
An alternative method to prove Proposition \ref{prop2} involves building an explicit flow composed with straight lines as in the proof of Proposition \ref{prop dim=d}.
However, for such method we need to assume that  $\omega$ is convex, contrarily to the more general approach developed in the proof of Proposition \ref{prop2}.
\end{rmq}

We now have all the tools to prove Theorem \ref{th cont approx}. 

\begin{proof}[Proof of Theorem \ref{th cont approx}]
Consider $\mu^0,\mu^1$ satisfying Condition  \ref{cond1}. 
%
By Lemma \ref{lemma cond}, there exist $T_0^*,~T_1^*, ~\omega_0$ for which  $\mu^0,\mu^1$ satisfy Condition \ref{cond2}. 
Let $\delta,~\varepsilon>0$ and $T:=T_0^*+T_1^*+\delta$.
We now prove that we can construct a Lipschitz uniformly bounded 
and  control $\mathds{1}_{\omega}u$
 such that the corresponding solution $\mu$ to System \eqref{eq:transport} satisfies
\vspace*{-2mm}\begin{equation*}
W_1(\mu(T),\mu^1)\leqslant \varepsilon.
\vspace*{-2mm}\end{equation*}

Denote by $T_0:=0$, $T_1:=T_0^*$,  $ T_2:=T_0^*+\delta/3$, $T_3:=T_0^*+2\delta/3$, $T_4:=T_0^*+\delta$ and  $ T_5:=T_0^*+T_1^*+\delta$. 
Also fix an open hypercube $S\subset\subset \omega\backslash\omega_0$. 
There exist $R>0$ and $\overline{x}\in\mb{R}^d$ such that the supports of $\mu^0$ and $\mu^1$ are strictly included in a hypercube with edges of length $R$. 
Define $$\overline{R}:=R+T\times\sup_{\mb{R}^d}|v|.$$
Applying Proposition \ref{prop1} on $[T_0,T_1]\cup[T_4,T_5]$ and Proposition \ref{prop2} on $[T_1,T_2]\cup[T_3,T_4]$, 
we can construct some space-dependent controls $u^1,~u^2,~u^4,~u^5$  Lipschitz and uniformly bounded, with $\supp(u^i)\subset\omega$, and two Borel sets $A_0$ and $A_1$ such that
$$\mu^0(A_0)=\mu^1(A_1)=\dfrac{\varepsilon}{2d\overline{R}}.$$
Moreover, the solution forward in time to
\begin{equation*}
	\left\{
	\begin{array}{ll}
\partial_t\rho_0 +\nabla\cdot((v+\mathds{1}_{\omega}u^1)\rho_0)=0
&\mbox{ in }\mb{R}^d\times[T_0,T_1],\\\noalign{\smallskip}
\partial_t\rho_0 +\nabla\cdot((v+\mathds{1}_{\omega}u^2)\rho_0)=0
&\mbox{ in }\mb{R}^d\times[T_1,T_2],\\\noalign{\smallskip}
\rho_0(T_0)=\mu^0_{|A_0^c}&\mbox{ in }\mb{R}^d
	\end{array}
	\right.
\end{equation*}
and the solution backward in time to
\begin{equation*}
	\left\{
	\begin{array}{ll}
\partial_t\rho_1 +\nabla\cdot((v+\mathds{1}_{\omega}u^5)\rho_1)=0
&\mbox{ in }\mb{R}^d\times[T_4,T_5],\\\noalign{\smallskip}
\partial_t\rho_1 +\nabla\cdot((v+\mathds{1}_{\omega}u^4)\rho_1)=0
&\mbox{ in }\mb{R}^d\times[T_3,T_4],\\\noalign{\smallskip}
\rho_1(T_5)=\mu^1_{|A_1^c}&\mbox{ in }\mb{R}^d
	\end{array}
	\right.
\end{equation*}
satisfy $\Supp(\rho_0(T_2))\subset S$ and $\Supp(\rho_1(T_3))\subset S.$ 

We remark that
$|\rho_0(T_2)|=|\rho_1(T_3)|=1-\varepsilon/2d\overline{R}$. 
We now apply Proposition \ref{prop dim=d} to approximately steer $\rho_0(T_2)$ to $\rho_1(T_3)$ inside $S$ as follows: we find a control $u^3$ on the time interval $[T_2,T_3]$ satisfying $\supp(u^3)\subset S$ such that the solution $\rho$ to
\begin{equation*}
	\left\{
	\begin{array}{ll}
\partial_t\rho +\nabla\cdot((v+\mathds{1}_{\omega}u^3)\rho)=0
&\mbox{ in }\mb{R}^d\times[T_2,T_3],\\\noalign{\smallskip}
\rho(T_2)=\rho_0(T_2)&\mbox{ in }\mb{R}^d
	\end{array}
	\right.
\end{equation*}
satisfies $$W_1(\rho(T_3),\rho_1(T_3))\leq \frac{\varepsilon}{2e^{2L(T_5-T_3)}},$$
where $L$ is the uniform Lipschitz constant for $u^4$ and $u^5$. 
Thus, denoting by $u$ the 
concatenation of $u^1$, $u^2$, $u^3$, $u^4$, $u^5$ on the time interval $[0,T]$, we approximately steer 
$\mu^0_{|A_0^c}$ to $\mu^1_{|A_1^c}$, since by \eqref{ine wasser 2} the solution 
$\mu$ to 
\begin{equation*}
	\left\{
	\begin{array}{ll}
\partial_t\mu +\nabla\cdot((v+\mathds{1}_{\omega}u^i)\mu)=0
&\mbox{ in }\mb{R}^d\times[T_{i-1},T_i],i\in\{1,...,5\},\\\noalign{\smallskip}
\mu(0)=\mu^0_{|A_0^c}&\mbox{ in }\mb{R}^d
	\end{array}
	\right.
\end{equation*}
satisfies
\begin{equation}\label{th1 eq1}
W_1(\Phi_T^{v+u}\#\mu^0_{|A_0^c},\mu^1_{|A_0^c})
=W_1(\mu(T_5),\mu^1_{|A_1^c})
\leq e^{2L(T_5-T_3)}\frac{\varepsilon}{2e^{2L(T_5-T_3)}}=\dfrac{\varepsilon}{2}.
\end{equation}
 Since we deal with AC measures, 
 using Properties \ref{prop Wp}, 
 there exists a measurable map
  $\gamma :\mb{R}^d\rightarrow\mb{R}^d$ such that
 \begin{equation*}
\left\{\begin{array}{l}
 \gamma\#\mu^1_{|A_1}=\Phi_T^{v+u}\#\mu^0_{|A_0},\\\noalign{\smallskip}
  W_1(\Phi_T^{v+u}\#\mu^0_{|A_0},\mu^1_{|A_1})
=\displaystyle\int_{\mb{R}^d}|x-\gamma(x)|d \mu^1_{|A_1}(x).
\end{array}\right.
\end{equation*}  
We deduce that
\begin{equation}\label{th1 eq2}
\begin{array}{rcl}
W_1(\Phi_T^{v+u}\#\mu^0_{|A_0},\mu^1_{|A_1})&=& \displaystyle\int_{\mb{R}^d}|x-\gamma(x)|d\mu^1_{|A_1}(x)
\leqslant  d\overline{R}\times\dfrac{\varepsilon}{2d\overline{R}}=\dfrac{\varepsilon}{2}.
\end{array}
\end{equation}
Inequalities \eqref{ine wasser},  \eqref{th1 eq1} and \eqref{th1 eq2} leads to the conclusion:
$$W_1(\Phi_T^{v+u}\#\mu^0,\mu^1)
\leqslant W_1(\Phi_T^{v+u}\#\mu^0_{|A_0^c},\mu^1_{|A_1^c})
+W_1(\Phi_T^{v+u}\#\mu^0_{|A_0},\mu^1_{|A_1})\leqslant \varepsilon.$$
~
\end{proof}

\section{Exact controllability}\label{sec:exact cont}

In this section, we study  exact controllability for System \eqref{eq:transport}.
In Section \ref{sec:non exact cont}, we show that exact controllability of System \eqref{eq:transport} does not hold for Lipschitz or BV controls.
In Section \ref{sec:exact cont borel}, we prove Theorem \ref{th cont exact}, \textit{i.e.} exact controllability of System \eqref{eq:transport} 
with a Borel localized control under some geometrical conditions.

\subsection{Negative results for exact controllability}\label{sec:non exact cont}

In this section, we show that exact controllability does not hold in general for Lipschitz or BV controls.
We will see that topological aspects play a crucial role at this level.

%

\subsubsection*{a) Non exact controllability with Lipschitz controls}


As explained in the introduction,  
if we impose the classical Carath\'eodory condition of $\mathds{1}_{\omega}u:\mb{R}^d\times\mb{R}^+\rightarrow\mb{R}^d$ being uniformly bounded, Lipschitz in space and  measurable in time, 
then the flow $\Phi^{v+\mathds{1}_{\omega}u}_t$ is a homeomorphism (see \cite[Th. 2.1.1]{BP07}).
More precisely, the flow and its inverse are locally Lipschitz.
This implies that the support of $\mu^0$ and $\mu(T)$ are homeomorph.
Thus, if the support of  $\mu^0$ and $\mu^1$ are not homeomorph, 
then exact controllability does not hold with Lipschitz controls. 
In particular, we cannot steer a measure which support is connected to a measure which support is composed of two connected components with Lipschitz controls and conversely.

\subsubsection*{b) Non exact controllability with BV controls}


To hope to obtain  exact controllability of System \eqref{eq:transport},
it is then  necessary to search for a control with less regularity. 
A weaker condition on the regularity of the vector field 
for the well-posedness of System \eqref{eq:transport} has been given in \cite{A04}.
Consider an initial data $\mu^0$ in $L^{\infty}\cap \mc{P}^{ac}_c(\mb{R}^d)$,
a target $\mu^1$ in $L^{\infty}\cap \mc{P}^{ac}_c(\mb{R}^d)$
and a vector field $u$ satisfying:
\begin{enumerate}
\item[(i)]$u(\cdot,t)\in BV_{loc}(\mb{R}^d)$ for a.e. $t\in(0,T)$.
\item[(ii)]For all $R>0$, 
\vspace*{-2mm}\begin{equation*}
\begin{array}{c}\sup\limits_{\mb{R}^d} |u|+\displaystyle\int_0^T\|[\Div u]^-\|_{L^{\infty}(B_R(0))}+\int_0^T\|u\|_{BV(B_R(0))}dt\\\hspace*{3cm}
+\displaystyle\int_0^T\int_{B_R(0)}|\Div u(t)|dxdt<+\infty.\end{array}
\end{equation*}
\end{enumerate}
Under these assumptions, there exists a unique solution to System \eqref{eq:transport} in the sense of \cite{A04}
(see \cite[Theorems 4.1 and 6.2]{A04}).

We now give an example of non exact controllability of System \eqref{eq:transport} 
with a velocity field satisfying these assumptions and $\omega=\mb{R}^d$. 
Consider $$\mu^0:=\frac12\mathds{1}_{(-1,0)}(x)dx+\frac12\mathds{1}_{(1,2)}(x)dx\mbox{~~ and ~~}\mu^1:=\frac12\mathds{1}_{(-1,1)}dx.$$ 
Suppose that there exists $u$ satisfying (i), (ii) 
and $\Phi_T^u\#\mu^0=\mu^1$.
The solution to System \eqref{eq:transport} is then unique.
Consider $y$ and $z$ solutions to 
\begin{equation*}
\left\{\begin{array}{l}
\dot y(t) =u(y(t),t)\mbox{ for a.e. } t\in(0,T),\\\noalign{\smallskip}
y(0)=0
\end{array}\right.
\mbox{ ~~and~~}
\left\{\begin{array}{l}
\dot z(t) =u(z(t),t)\mbox{ for a.e. } t\in(0,T),\\\noalign{\smallskip}
z(0)=1.
\end{array}\right.
\end{equation*}
Since $\Phi_T^u((-1,0))=(-1,0)$ and $\Phi_T^u((1,2))=(0,1)$, it holds $y(T)=z(T)$.
Thus there exists an infimum time $t_1\in[0,T]$   such that $y(t_1)=z(t_1)$. 
Since $\Phi_T^u\#\mu^0=\mu^1$, we obtain a contradiction with (ii) by observing that
 there exists $t_0\in (0,t_1)$ such that
$$\displaystyle\int_0^T\| [\partial_x u(t)]^-\|_{\infty}dt
\geqslant \displaystyle\int_{t_0}^{t_1}\left|\dfrac{u(y(t),t)-u(z(t),t)}{y(t)-z(t)}\right|dt
= \displaystyle\int_{t_0}^{t_1}\left|\dfrac{\dot y(t)-\dot z(t)}{y(t)-z(t)}\right|dt=+\infty.$$

Thus, we cannot steer 
a measure which support is composed of two connected components 
 to a measure which support is connected
 with BV controls satisfying (i) and (ii), 
hence general results on exact controllability cannot hold.
However, the inverse is possible.
For example, if we denote by
$$\mu^0:=\mathds{1}_{(-1,1)}dx\mbox{ and } 
u(x):=\left\{\begin{array}{ll}
\sqrt{x}&\mbox{ if }x\geqslant 0,\\
0&\mbox{ otherwise},
\end{array}\right.$$
then $u$ satisfies (i) and (ii) and the unique solution $\mu$ to System \eqref{eq:transport} 
is given by
$$\mu(t)=\mathds{1}_{(-1,0)}(x)dx+\left(1-\frac{t}{2\sqrt{x}}\right)\times
\mathds{1}_{(\frac{t^2}{4},(\frac{t}{2}+1)^2)}(x)dx.$$

\subsection{Exact controllability with Borel controls}\label{sec:exact cont borel}

In this section, we prove Theorem \ref{th cont exact}, \textit{i.e.}  exact controllability of System 
\eqref{eq:transport} in the following sense: there exists a couple $(\mathds{1}_{\omega}u,\mu)$   solution to System \eqref{eq:transport} satisfying $\mu(T)=\mu^1$. 
Before proving Theorem \ref{th cont exact}, we need three useful results.

The first one is the following proposition, showing that we can store
the whole mass of $\mu^0$ in $\omega$, under Condition \ref{cond2}.
It is the analogue of Proposition \ref{prop1}. In this case, we control the whole mass, 
but we do not have necessarily uniqueness of the solution to System \eqref{eq:transport}.

\begin{prop}\label{prop1 exact}
Let $\mu^0\in\mc{P}_c(\mb{R}^d)$ satisfying the first item of Condition \ref{cond2}.
Then there exists a couple $(\mathds{1}_{\omega}u,\mu)$ composed of  
a Borel vector field $\mathds{1}_{\omega}u:\mb{R}^d\times\mb{R}^+\rightarrow\mb{R}^d$ 
and  a time-evolving measure $\mu$
being weak  solution to  System \eqref{eq:transport} 
 and satisfying  
\vspace*{-2mm} \begin{equation*}
 \supp(\mu(T_0^*))\subset\omega.
 \end{equation*}
\end{prop}

\begin{proof}
 For each $x^0\in\mb{R}^d$, we denote by 
$$\widetilde{t}^0(x^0):=\inf\{t\geqslant 0:\Phi_t^v(x^0)\in \overline{\omega}_0\}$$
and consider the application $\Psi_{\cdot}(x^0)$ defined for all $t\geqslant 0$ by 
$$\Psi_t(x^0)=\left\{\begin{array}{ll}
\Phi^{v}_t(x^0)&\mbox{ if }t\leqslant \widetilde{t}^0(x^0),\\\noalign{\smallskip}
\Phi^{v}_{\widetilde{t}^0(x^0)}(x^0)&\mbox{ otherwise.}\end{array}\right.$$
For all $t\geqslant 0$, the application $\Psi_t$ 
is a Borel map. Consider $\mu$ defined for all $t\geqslant 0$ by
$$\mu(t):=\Psi_t\#\mu^0.$$
We remark that, for all $t,s\in[0,T_0^*]$ such that 
$t\geqslant s$, 
\begin{equation}\label{push}
\mu(t)=\Psi_{t-s}\#\mu(s).
\end{equation}
Since $\Phi_{\cdot}^v(x^0)$ is Lipschitz, for all $x^0\in\mb{R}^d$ and $t\in[0,T^*_0]$, it holds
\begin{equation}\label{psi lip}
|\Psi_t(x^0)-x^0|\leqslant C\min\{t,t^0(x^0)\}\leqslant Ct.
\end{equation}
Combining \eqref{push} and \eqref{psi lip}, we deduce for all $t,s\in[0,T_0^*]$ with $s\leqslant t$
\begin{equation*}
W_2^2(\mu(s),\mu(t))
\leqslant \displaystyle\int_{\mb{R}^d}|\Psi_{t-s}(x)-x|^2~d\mu(s)\leqslant 
\sup_{x\in\mb{R}^d}|\Psi_{t-s}(x)-x|^2\leqslant C|t-s|^2.
\end{equation*}
We deduce that the metric derivative $|\mu'|$ of $\mu$ defined for all $t\in[0,T_0^*]$ by 
\begin{equation}\label{mu'}
|\mu'|(t):=\lim\limits_{s\rightarrow t}\dfrac{W_2(\mu(t),\mu(s))}{|t-s|}\end{equation}
is uniformly bounded on $[0,T_0^*]$.
Then $\mu$ is an absolute continuous curve on $\mc{P}_c(\mb{R}^d)$ (see \cite[Def. 1.1.1]{AGS05}).
Using \cite[Th. 8.3.1]{AGS05}, there exists a Borel vector $w:\mb{R}^d\times(0,T^*_0)\rightarrow\mb{R}^d$
satisfying 
$$\|w(t)\|_{L^2(\mu(t);\mb{R}^d)}\leqslant |\mu'|(t)\mbox{ a.e. }t\in[0,T^*_0]$$
and the couple $(w,\mu)$ is a weak solution to 
\begin{equation}\label{weak ambr}
	\left\{
	\begin{array}{ll}
\partial_t\mu +\nabla\cdot(w\mu)=0
&\mbox{ in }\mb{R}^d\times[0,T^*_0],\\\noalign{\smallskip}
\mu(0)=\mu^0&\mbox{ in }\mb{R}^d.
	\end{array}
	\right.
\end{equation}
Moreover, for all $t\in[0,T_0^*]$, it holds
$$w(t)\in \mbox{Tan}_{\mu(t)}(\mc{P}_c(\mb{R}^d))
:=\overline{\{\nabla\varphi:\varphi\in\mc{C}_c^{\infty}(\mb{R}^d)\}}^{L^2(\mu(t);\mb{R}^d)}.$$

Consider an open set $\omega_1$ of class $\mc{C}^{\infty}$ 
satisfying $\omega_0\subset\subset\omega_1\subset\subset\omega$.
We now prove that $w(t)$ coincides with $v(t)$ in $\supp(\mu(t))\backslash\overline{\omega}_1$ \textit{a.e.} $t\in[0,T_0^*]$,
\textit{i.e.} we can choose $u=0$ outside $\omega$.
Fix $t\in[0,T_0^*]$  and consider $x\in \supp(\mu(t))\cap\omega_1^c$.
There necessarily exists $x^0\in\supp(\mu^0)$ such that $\Phi^v_t(x^0)=x$, otherwise $x\in\partial\omega_0$.
Moreover for a $B:=B_r(x^0)$ with $r>0$ 
$\Phi_s^v(B)\subset\subset \omega_0^c$
for all $s\in[0,t]$, 
otherwise there exists $s\in[0,t]$ for which $\Phi^v_s(x^0)\in\partial\omega_0$.
Thus \begin{equation}\label{phi=psi}
\Phi_t^v=\Psi_t\mbox{ in }B.
\end{equation}
 We denote by $A:=\Phi_t^v(B)$.
%
%
We have now prove that 
\begin{equation}\label{psi A=phi A}
\Psi_t^{-1}(A)=(\Phi_t^v)^{-1}(A).
\end{equation}
Consider $x\in (\Phi_t^v)^{-1}(A)$. 
Equality \eqref{phi=psi} implies $\Phi_t^v(x)=\Psi_t(x)$. Then $x\in \Psi_t^{-1}(A)$.
Consider now $x\in \Psi_t^{-1}(A)$, which means $\Psi_t(x)\in A$. 
Using the fact that $A\cap\overline{\omega}_0\neq0$, $t<\widetilde{x}^0(x)$. 
Then $\Psi_t(x)=\Phi_t^v(x)$ and $x\in (\Phi_t^v)^{-1}(A)$.
Thus \eqref{psi A=phi A} holds.
%
By definition of the push forward,
$$\mu_{|A}(t)=\Psi_t\#(\mu^0_{|\Psi_t^{-1}(A)})
\mbox{ and }
(\Phi_t^v\#\mu^0)_{|A}=\Phi_t^v\#(\mu^0_{|\Phi_t^{-1}(A)}).$$
Since $\Psi_t=\Phi_t^v$ on the set $B=(\Phi_t^v)^{-1}(A)=\Psi_t^{-1}(A)$, this implies 
$$\mu_{|A}(t)=\Phi^v_{t}\#\mu^0_{|A}.$$
By compactness of $\supp(\mu(t)\cap\omega_1^c$, it holds
$$\mu(t)_{|\omega_1^c}=(\Phi^v_{t}\#\mu^0)_{|\omega_1^c}.$$
We deduce that, for all $\varphi\in\mc{C}^{\infty}_c(\mb{R}^d)$ 
such that  $\supp(\varphi)\subset\subset\omega_1^c$,
$$\dfrac{d}{dt}\displaystyle\int_{\mb{R}^d}\varphi~d\mu(t)
=\displaystyle\int_{\mb{R}^d}\langle\nabla\varphi,w\rangle~d\mu(t)
\mbox{~~and~~}
\dfrac{d}{dt}\displaystyle\int_{\mb{R}^d}\varphi~d\mu(t)
=\displaystyle\int_{\mb{R}^d}\langle\nabla\varphi,v\rangle~d\mu(t).$$

If it holds
$v\in \mbox{Tan}_{\mu(t)}(\mc{P}_c(\mb{R}^d))$,
then $w(t)=v$, $\mu(t)$ a.e. in $\overline{\omega_1}^c$,
and we conclude by taking $u:=w-v$ which is supported in $\omega$.

If now $v\not\in \mbox{Tan}_{\mu(t)}(\mc{P}_c(\mb{R}^d))$, 
we can write $v=v_1+v_2$ with $v_1\in \mbox{Tan}_{\mu(t)}(\mc{P}_c(\mb{R}^d))$ and $v_2\in \mbox{Tan}_{\mu(t)}(\mc{P}_c(\mb{R}^d))^{\perp}$,
where 
$$\mbox{Tan}_{\mu(t)}(\mc{P}_c(\mb{R}^d))^{\perp}
=\{\nu\in L^2(\mu(t):\mb{R}^d):\nabla\cdot(\nu\mu(t))=0\}$$
(see for instance \cite[Prop. 8.4.3]{AGS05}). In other terms, $v_2$ plays no role in the weak formulation of the continuity equation.
Thus, with the same argument, we can prove that $w(t)=v_1$, $\mu(t)$ a.e. in $\overline{\omega_1}^c$
and we conclude by tacking $u:=w-v_1$.
\end{proof}

The second useful result for the proof of Theorem \ref{th cont exact} allows 
to exactly steer a measure contained in $\omega$ to  a nonempty open convex set $S\subset\subset \omega$.
It is the analogue of Proposition \ref{prop2}. In this case, 
as in Proposition \ref{prop1 exact}, we control the whole mass, 
but we do not have necessarily uniqueness of the solution to System \eqref{eq:transport}.

\begin{prop}\label{prop2 exact}
Let $\mu^0\in\mc{P}_c(\mb{R}^d)$ satisfying $\supp(\mu^0)\subset\omega$.
Define a nonempty open convex set $S$ strictly included in $\omega\backslash\supp(\mu^0)$ and choose $\delta>0$.
Then there exists a couple $(\mathds{1}_{\omega}u,\mu)$ composed of  
a Borel vector field $\mathds{1}_{\omega}u:\mb{R}^d\times\mb{R}^+\rightarrow\mb{R}^d$ 
and  a time-evolving measure $\mu$
being weak  solution to  System \eqref{eq:transport} 
  satisfying  
  \vspace*{-2mm}\begin{equation*}
  \supp(\mu(\delta))\subset S.
  \end{equation*}
\end{prop}

\begin{proof}
Consider $S_0$ a nonempty open set of $\mb{R}^d$ of class $\mc{C}^{\infty}$ strictly included in $S$ and 
 $\omega_1$ an open set of $\mb{R}^d$ of class $\mc{C}^{\infty}$ satisfying 
\vspace*{-2mm}\begin{equation*}
\supp(\mu^0)\cup S\subset\subset\omega_1\subset\subset\omega.
\vspace*{-2mm}\end{equation*}
An example is given in Figure \ref{fig:constr S}.
Consider $\eta \in\mathcal{C}^2(\overline{\omega_1})$ defined in the proof of Proposition \ref{prop2} 
satisfying \eqref{prop eta}.
For all $k\in\mb{N}^*$, we consider a Lipschitz vector field  $v_k$ satisfying
\vspace*{-2mm}\begin{equation*}
v_k:=\left\{\begin{array}{ll}
 k\nabla \eta &\mbox{in}~ \omega_1,\\
v&\mbox{in}~\omega^c.
\end{array}\right.
\vspace*{-2mm}\end{equation*}
We denote by 
$$\widetilde{t}^0_k(x^0):=\inf\{t\geqslant 0:\Phi_t^{v_k}(x^0)\in \overline{S}_0\}.$$
For all $x^0\in\mb{R}^d$ and all $k\in\mb{N}^*$, consider the application $\Psi_{k,\cdot}(x^0)$ defined for all $t\geqslant 0$ by 
$$\Psi_{k,t}(x^0)=\left\{\begin{array}{ll}
\Phi^{v_k}_t(x^0)&\mbox{ if }t\leqslant \widetilde{t}^0_k(x^0),\\\noalign{\smallskip}
\Phi^{v_k}_{\widetilde{t}^0_k(x^0)}(x^0)&\mbox{ otherwise.}\end{array}\right.$$
Using the same argument as in the proof of Proposition \ref{prop2}, 
for $K$ large enough, $\Psi_{K,\delta}(x^0)$ belongs to $S$ for all $x^0\in\supp(\mu^0)$.
Consider $\mu$ defined for all $t\in(0,\delta)$ by $\mu(t):=\Psi_{K,t}\#\mu^0$. 
As in the proof of Proposition \ref{prop1 exact}, there exists a vector field $u_K$ 
such that $(u_K,\mu)$ is a weak solution to System \eqref{weak ambr}. Moreover 
$u_K(t)=v_K$, $\mu(t)$ a.e.  in $\overline{S}^c$ and a.e. $t\in[0,\delta]$.
Thus, we conclude that  $(\mathds{1}_{\omega}(u_K-v_K),\mu)$ is solution to System \eqref{eq:transport} and $\supp(\mu(\delta))\subset S$.
\end{proof}

The third useful result for the proof of Theorem \ref{th cont exact} allows 
to exactly steer a measure contained in a nonempty open convex set $S\subset\subset \omega$ to a given measure contained in $S$.
It is the analogue of Proposition \ref{prop dim=d}. In this situation, we obtain exact controllability of System \eqref{eq:transport},
but, again, we do not have necessarily uniqueness of the solution to System \eqref{eq:transport}.

\begin{prop}\label{prop3 exact}
Let $\mu^0,~\mu^1\in\mc{P}_c(\mb{R}^d)$ satisfying $\supp(\mu^0)\subset S$ and $\supp(\mu^1)\subset S$
for a nonempty open convex set $S$ strictly included in $\omega$.  Choose $\delta>0$.
Then there exists a couple $(\mathds{1}_{\omega}u,\mu)$ composed of  
a Borel vector field $\mathds{1}_{\omega}u:\mb{R}^d\times\mb{R}^+\rightarrow\mb{R}^d$ 
and  a time-evolving measure $\mu$
being weak  solution to  System \eqref{eq:transport} 
 and satisfying  
\vspace*{-2mm} \begin{equation*}
 \supp(\mu)\subset S\mbox{ and }\mu(\delta)=\mu^1.
 \end{equation*}
\end{prop}

\begin{proof}
Let $\pi$ be the optimal plan given in \eqref{def:wass plan} 
associated to the Wasserstein distance between $\mu^0$ and $\mu^1$. 
For $i\in\{1,2\}$, we denote by $p_i:\mb{R}^d\times\mb{R}^d\rightarrow\mb{R}^d$ the projection operator defined by
$$p_i:(x_1,x_2)\mapsto x_i.$$
 Consider the time-evolving measure $\mu$ defined for all $t\in[0,\delta]$ by
\begin{equation}\label{rho inter}
\mu(t):=\dfrac{1}{\delta}\left[(\delta-t)p_1 + tp_2\right]\#\pi.
\end{equation}
Using \cite[Th. 7.2.2]{AGS05}, 
$\mu$ is a constant speed geodesic connecting $\mu^0$ and $\mu^1$
in  $\mc{P}_c(\mb{R}^d)$,
\textit{i.e.} for all $s,t\in[0,\delta]$
\vspace*{-2mm}\begin{equation*}
W_2(\mu(t),\mu(s))= \dfrac{(t-s)}{\delta}W_2(\mu^0,\mu^1).
\end{equation*}
We deduce that the metric derivative $|\mu'|$ of $\mu$ 
(see \eqref{mu'})
is uniformly bounded on $[0,\delta]$.
Then $\mu$ is an absolute continuous curve on $\mc{P}_c(\mb{R}^d)$ (see \cite[Def. 1.1.1]{AGS05}).
Thus, using  \cite[Th. 8.3.1]{AGS05},
 there exists a Borel vector field 
 $w:\mb{R}^d\times(0,\delta)\rightarrow\mb{R}^d$
such that 
$$\|w(t)\|_{L^2(\mu(t);\mb{R}^d)}\leqslant |\mu'|(t)\mbox{ a.e. }t\in[0,\delta]$$
and the couple $(w,\mu)$ is a weak solution to 
\begin{equation*}
	\left\{
	\begin{array}{ll}
\partial_t\mu +\nabla\cdot(w\mu)=0
&\mbox{ in }\mb{R}^d\times[0,\delta],\\\noalign{\smallskip}
\mu(0)=\mu^0&\mbox{ in }\mb{R}^d.
	\end{array}
	\right.
\end{equation*}
Consider $\theta\in\mc{C}^{\infty}_c(\mb{R}^d)$ such that 
\begin{equation*}
0\leqslant \theta\leqslant 1,~~
\theta = 1\mbox{ in }S\mbox{~~and~~}
\theta = 0\mbox{ in }\omega^c.
\end{equation*} 
We remark that $\mu$ is supported in $S$, 
then the couple $(\mathds{1}_{\omega}u,\mu)$ 
with 
\vspace*{-2mm}\begin{equation*}
u:=\theta\times (w-v)
\vspace*{-2mm}\end{equation*} 
is solution to 
\vspace*{-2mm}\begin{equation*}
	\left\{
	\begin{array}{ll}
\partial_t\mu +\nabla\cdot((v+\mathds{1}_{\omega}u)\mu)=0
&\mbox{ in }\mb{R}^d\times[0,\delta],\\\noalign{\smallskip}
\mu(0)=\mu^0&\mbox{ in }\mb{R}^d.
	\end{array}
	\right.
\end{equation*}
$\left.\right.$
\end{proof}

\vspace*{-2mm} \begin{rmq}
It is well know (see for instance \cite{S15}) that 
for two measures compactly supported with the same total mass $\mu^0,\mu^1$, 
the Wasserstein distance can be express as follows
\vspace*{-2mm}\begin{equation}\label{eq:benamou brenier}
\begin{array}{c}W_2(\mu^0,\mu^1)
=\min\limits_{(v,\mu)\in\mc{B}}\left\{\left(\displaystyle\int_0^1\int_{\mb{R}^d}|v(t)|^2d\mu(t)dt\right)^{1/2}:
\right.\hspace*{3cm}\\\hspace*{3cm}\left.
\partial_t\mu+\nabla\cdot(v\mu)=0,~\mu(0)=\mu^0,~\mu(1)=\mu^1\right\},\end{array}
\end{equation}
where $\mc{B}$ is the set of couples $(v,\mu)$ composed of a time evolving measure $\mu(t)$ 
and a Borel vector field $v:\mb{R}^d\times\mb{R}\rightarrow\mb{R}^d$ satisfying 
\vspace*{-2mm}\begin{equation*}\int_0^1\int_{\mb{R}^d}|v(t)|^2d\mu(t)dt<\infty.\end{equation*}
Equality \eqref{eq:benamou brenier} is called the \textit{Benamou-Brenier Formula}.
In the proof of the Proposition \ref{prop3 exact}, it is possible to replace the definition in \eqref{rho inter} of $\mu$ by the minimizer of \eqref{eq:benamou brenier}
which already satisfies the continuity equation.
\end{rmq}

We now have all the tools to prove Theorem \ref{th cont exact}.
\begin{proof}[Proof of Theorem \ref{th cont exact}]
Consider $\mu^0$ and $\mu^1$ satisfying Condition \ref{cond1}.
Applying Lemma \ref{lemma cond}, Condition \ref{cond2} holds for some $\omega_0$, $T_0^*$ and $T_1^*$.
Let $T:=T_0^*+T_1^*+\delta$ with $\delta>0$ and  $T_0,~T_1,~T_2,~T_3,~T_4,~T_5$ be the times given in the proof of Theorem \ref{th cont approx}.
Using  Proposition \ref{prop1 exact} on $[T_0,T_1]\cup[T_4,T_5]$, 
there exist $\rho_1\in\mc{C}^0([T_0,T_1],\mc{P}_c(\mb{R}^d))$, $\rho_5\in\mc{C}^0([T_4,T_5],\mc{P}_c(\mb{R}^d))$
and some space-dependent Borel controls $u^1,~u^5$  
 with 
 \vspace*{-2mm}\begin{equation*}
 \supp(u^1)\cup\supp(u^5)\subset\omega
\vspace*{-2mm} \end{equation*}
  such that $(\mathds{1}_{\omega}u^1,\rho_1)$
  is a weak solution forward in time to
\vspace*{-2mm}\begin{equation*}
	\left\{
	\begin{array}{ll}
\partial_t\rho_1 +\nabla\cdot((v+\mathds{1}_{\omega}u^1)\rho_1)=0
&\mbox{ in }\mb{R}^d\times[T_0,T_1],\\\noalign{\smallskip}
\rho_1(T_0)=\mu^0&\mbox{ in }\mb{R}^d
	\end{array}
	\right.
\vspace*{-2mm}\end{equation*}
and $(\mathds{1}_{\omega}u^5,\rho_5)$ is a weak solution backward in time to
\vspace*{-2mm}\begin{equation*}
	\left\{
	\begin{array}{ll}
\partial_t\rho_5 +\nabla\cdot((v+\mathds{1}_{\omega}u^5)\rho_5)=0
&\mbox{ in }\mb{R}^d\times[T_4,T_5],\\\noalign{\smallskip}
\rho_5(T_5)=\mu^1&\mbox{ in }\mb{R}^d.
	\end{array}
	\right.
\vspace*{-2mm}\end{equation*}
Moreover $\Supp(\rho_1(T_1))\subset \omega$ and $\Supp(\rho_5(T_4))\subset \omega.$ 
Consider a nonempty open convex set $S$ strictly included in $\omega\backslash\omega_0$.
Using  Proposition \ref{prop2 exact} on $[T_1,T_2]\cup[T_3,T_4]$, 
there exist $\rho_2\in\mc{C}^0([T_1,T_2],\mc{P}_c(\mb{R}^d))$, $\rho_4\in\mc{C}^0([T_3,T_4],\mc{P}_c(\mb{R}^d))$
and some space-dependent Borel controls $u^2,~u^4$  
 with $$\supp(u^2)\cup\supp(u^4)\subset\omega$$ such that $(\mathds{1}_{\omega}u^2,\rho_2)$ 
 is a weak solution forward in time to
\vspace*{-2mm}\begin{equation*}
	\left\{
	\begin{array}{ll}
\partial_t\rho_2 +\nabla\cdot((v+\mathds{1}_{\omega}u^2)\rho_2)=0
&\mbox{ in }\mb{R}^d\times[T_1,T_2],\\\noalign{\smallskip}
\rho_2(T_1)=\rho_1(T_1)&\mbox{ in }\mb{R}^d
	\end{array}
	\right.
\vspace*{-2mm}\end{equation*}
and $(\mathds{1}_{\omega}u^4,\rho_4)$ is a weak solution backward in time to
\vspace*{-2mm}\begin{equation*}
	\left\{
	\begin{array}{ll}
\partial_t\rho_4 +\nabla\cdot((v+\mathds{1}_{\omega}u^4)\rho_4)=0
&\mbox{ in }\mb{R}^d\times[T_3,T_4],\\\noalign{\smallskip}
\rho_4(T_4)=\rho_5(T_4)&\mbox{ in }\mb{R}^d.
	\end{array}
	\right.
\vspace*{-2mm}\end{equation*}
Moreover $\Supp(\rho_2(T_2))\subset S$ and $\Supp(\rho_4(T_3))\subset S.$ 
Using  Proposition \ref{prop3 exact} on $[T_2,T_3]$, 
there exist $\rho_3\in\mc{C}^0([T_2,T_3],\mc{P}_c(\mb{R}^d))$ satisfying $\supp(\rho_3)\subset S$
and a Borel control $u^3$  
 with $$\supp(u^3)\subset\omega$$ such that $(\mathds{1}_{\omega}u^3,\rho_3)$ is a weak solution forward in time to
\begin{equation*}
	\left\{
	\begin{array}{ll}
\partial_t\rho_3 +\nabla\cdot((v+\mathds{1}_{\omega}u^3)\rho_3)=0
&\mbox{ in }\mb{R}^d\times[T_2,T_3],\\\noalign{\smallskip}
\rho_3(T_2)=\rho_2(T_2)&\mbox{ in }\mb{R}^d
	\end{array}
	\right.
\end{equation*}
and satisfies $\rho_3(T_3)=\rho_4(T_3)$.
Thus the couple $(\mathds{1}_{\omega}u,\mu)$ defined by
\begin{equation*}
(\mathds{1}_{\omega}u,\mu)
=(\mathds{1}_{\omega}u^i,\rho_i) \mbox{ in }\mb{R}^d\times [T_{i-1},T_i),~i\in\{1,...,5\}
\end{equation*}
is a weak solution to System \eqref{eq:transport} 
and satisfies $\mu(T)=\mu^1.$
\end{proof}

\section*{Acknowledgments} 
The authors thank F. Santambrogio 
  for his interesting comments and suggestions.

\bibliographystyle{plain}
\bibliography{references.bib}
\end{document}